\documentclass{amsart}

\usepackage{amssymb}
\usepackage{latexsym}
\usepackage{amsmath}
\usepackage[utf8]{inputenc}
\usepackage[english]{babel}
\usepackage[dvipsnames]{xcolor} 
\usepackage[colorlinks=true, citecolor=NavyBlue, linkcolor=Peach]{hyperref} 
\usepackage{tikz}
\usetikzlibrary{patterns}
\usepackage{graphicx}
\usepackage{subcaption}
\usepackage{setspace}
\usepackage[shortlabels]{enumitem} 
\usepackage{mathtools}
\usepackage{braket}
\usepackage{csquotes} 
\usepackage{comment} 

\usepackage[normalem]{ulem}

\usepackage[letterpaper, left=2.5cm, right=2.5cm, top = 1in, bottom = 1in]{geometry}

\theoremstyle{plain}
\newtheorem{thm}{Theorem}[section]
\newtheorem{lemma}[thm]{Lemma}
\newtheorem{cor}[thm]{Corollary}
\newtheorem{con}[thm]{Conjecture}
\newtheorem{prop}[thm]{Proposition}

\theoremstyle{definition}
\newtheorem{defn}[thm]{Definition}
\newtheorem{ex}[thm]{Example}
\newtheorem{remark}[thm]{Remark}
\newtheorem{notation}[thm]{Notation}
\newtheorem{assumption}[thm]{Assumption}

\DeclareMathOperator\lk{Lk} 
\DeclareMathOperator\Tk{Tk} 

\newcommand{\R}{\mathbb{R}}
\newcommand{\code}{\mathcal{C}} 
\newcommand{\simp}[1]{\Delta(#1)}
\newcommand{\mincode}{\code_{\text{min}}} 
\newcommand{\link}[2]{\lk_{#1}{(#2)}} 

\newcommand{\mcC}{\mathcal{C}} 
\newcommand{\mcW}{\mathcal{W}} 
\newcommand{\C}{\mathcal{C}}

\date{\today}

\begin{document}

\title{Convexity of Neural Codes with Four Maximal Codewords}

\author[Ahmed]{Saber Ahmed} 
\address{Saber~Ahmed\\ 
         Department of Mathematics and Statistics\\ 
         Hamilton College\\ 
         Clinton\\ 
         New York \ 13323\\ 
         USA} 
\author[Crepeau]{Natasha Crepeau} 
\address{Natasha~Crepeau\\ 
         Department of Mathematics\\ 
         University of Washington\\ 
         Seattle\\ 
         Washington \ 98195\\ 
         USA} 
\author[Flores]{Gisel Flores} 
\address{Gisel~Flores \\
    Department of Mathematics \\
    University of Wisconsin \\
    Madison \\
    Wisconsin \ 53706 \\
    USA} 
\author[Isekenegbe]{Osiano Isekenegbe} 
\address{Osiano~Isekenegbe \\
    Department of Mathematics \\
    Duke University \\
    Durham \\
    North Carolina \ 27708 \\
    USA} 
\author[Perez]{Deanna Perez} 
\address{Deanna~Perez \\
    Department of Mathematics \\
    California State University, Fresno \\
    Fresno \\
    California \ 93740
    USA} 

\author[Shiu]{Anne Shiu} 
\address{Anne~Shiu\\ 
         Department of Mathematics\\ 
         Texas A\&M University\\ 
         College Station\\ 
         Texas \ 77843\\ 
         USA} 
\begin{abstract}
Place cells are neurons that act as biological position sensors, associated with and firing in response to regions of an environment to situate an organism in space. These associations are recorded in (combinatorial) neural codes,
motivating the following mathematical question: Which neural codes are generated by a collection of convex open sets in Euclidean space? Giusti and Itskov showed that a necessary condition for convexity is the absence of ``local obstructions.'' This necessary condition is, in fact, sufficient for certain families of codes. 
One such family consists of all codes
with up to three maximal codewords. In this article, we investigate codes with four maximal codewords, showing that for many such codes, convexity is characterized by the absence of local obstructions, whereas for other such codes, convexity is characterized by the absence of local obstructions and a second type of obstruction, a ``wheel”. Key to our analysis is a case-by-case investigation based on the nerve complex of the set of maximal codewords of a neural code. Up to symmetry, there are 20 possible nerves; and our results fully characterize convexity in 15 of the 20 cases.

\vspace{.1in}
\noindent
 {\bf Keywords:} Neural code, simplicial complex, nerve, convex, local obstruction, wheel

 \vspace{.1in}
\noindent
{\bf MSC Codes:}
13F55, 
52A20, 
92C20 
\end{abstract}

\maketitle

\section{Introduction}
In the brain, \emph{place cells} act as position sensors. Each place cell is associated with a ``preferred'' region of the animal's environment called its \emph{place field} –- when the animal is in a place field, the corresponding place cell fires at a heightened rate. Place fields can be represented as subsets $U_i$ of Euclidean space $\R^d$ (typically, $d=2$ or $d=3$), where subset $U_i$ induces neuron $i$ to fire. The collection $\mathcal{U} = \{U_1,U_2,\dots,U_n\}$ then naturally assigns to each point $x \in \R^d$ a binary string $\sigma_x = c_1 c_2 \dots c_n \in \{0,1\}^n$, called a \emph{codeword}, such that $$ c_\ell = \begin{cases}
    1, & \text{ if } x \in U_\ell \\ 0, & \text{ otherwise}.
\end{cases}$$ A collection of such binary strings is called a \emph{(combinatorial) neural code}, which encodes the animal's neural firing activity. 

    A surprising feature of experimentally observed place fields is the predominance of convex or nearly convex shapes. Their boundaries are also gradual, rather than strictly delimited, reflecting diminishing firing rates farther from the field's center. As such, place fields are well modeled by convex open sets in Euclidean space. Accordingly, codes that can be realized by a collection of convex open sets -- we call such codes {\em convex} -- are of particular interest to neuroscientists.

Towards understanding when codes are \uline{not} convex, Giusti and Itskov introduced the combinatorial \emph{local obstruction} and showed that codes containing local obstructions are not convex~\cite{Giusti}. Subsequently, Ruys de Perez {\em et al.}\ introduced the concept of a \emph{wheel} and showed that codes containing wheels are not convex~\cite{wheels}. Local obstructions and wheels are two primary tools for 
establishing that a given code is not convex, and they are a focus of this article. 
(Additional tools can be found, for instance, in~\cite{morphisms, graphs-neural}.)

For codes with up to $3$ \emph{maximal codewords},
Johnston {\em et al.}\ proved that 
avoiding local obstructions is not only necessary but also sufficient for convexity~\cite{Shiu}.
Their result, however, does not extend for codes with $4$ or more maximal codewords~\cite{Lienkaemper}.  Nevertheless, for certain codes on up to $6$ neurons and $4$ maximal codewords, 
Ruys de Perez {\em et al.}\ showed that being convex is equivalent to avoiding both local obstructions and wheels~\cite{wheels}.  
Subsequently, Amzi Jeffs conjectured that this result generalizes to all codes (with any number of neurons) having at most $4$ maximal codewords (personal communication), as follows.

\begin{con}{(Amzi Jeffs)} \label{conj-Jeffs}
Let $\C$ be a code with up to 4 maximal codewords. Then $\C$ is convex if and only if $\C$ has no local obstructions and no wheels. 
\end{con}

This article makes significant progress toward resolving Conjecture~\ref{conj-Jeffs}, as summarized in Table~\ref{tab:summary_of_results}.  Our approach is to split all codes with $4$ maximal codewords into $20$ groups based on how the maximal codewords intersect each other, or in other words, based on the \emph{nerve} of the set of maximal codewords. 
In other words, these $20$ groups correspond to the possible simplicial complexes on exactly $4$ vertices; these are labeled by L9 through L28 in Figure~\ref{fig:image-simplicial-complex} (L1 through L8 correspond to simplicial complexes with up to $3$ vertices).  One of our main results (Corollary~\ref{cor:conjecture}) states that Conjecture~\ref{conj-Jeffs} holds for all cases from L9 to L23 (with prior results, as mentioned earlier, handling cases L1 to L8).

The unresolved cases therefore are those from L24 to L28. 
Nonetheless, we are able to resolve some but not all cases of L24 (Theorem~\ref{thm:L24-summary}).  Additionally, we prove a general result that asserts that Conjecture~\ref{conj-Jeffs} extends to all codes (with any number of maximal codewords) for which the nerve of the set of maximal codewords has no $2$-simplices (that is, consists only of vertices and edges); see Proposition~\ref{prop:notriangle}.

\begin{table}[]
    \noindent\makebox[\textwidth]{
    \begin{tabular}{|ccc|} \hline
      Case   & Convexity criterion & Result \\ \hline
      L1--L8
        & Convex $\Leftrightarrow$ no local obstructions
        & 
        \cite{Shiu}
        (see Proposition~\ref{prop:3-max})
        \\
      L9--L12, L15--L16   & Convex $\Leftrightarrow$ no local obstructions & 
      Proposition~\ref{prop:disconnected-nerve}
      \\ 
      L13, L14, L17, L19, L20, L23 &  Convex $\Leftrightarrow$ no local obstructions  & Corollary~\ref{cor:no-triangle}\\
      L18, L21 & Convex $\Leftrightarrow$ no local obstructions & Theorem \ref{thm: l18-l21}  \\
      L22 & Convex $\Leftrightarrow$ no local obstructions & Theorem \ref{thm: l22} \\
      L24 -- minimal codes & Convex $\Leftrightarrow$ no local obstructions \& no wheels & Theorem \ref{thm: sprocket-L24}
      \\
      \hline
    \end{tabular}
    }
    \caption{Summary of main results.  This table lists the cases of neural codes with up to four maximal codewords for which we can characterize convexity.  These cases are listed by the nerve of the set of maximal codewords (with labels as in Figure~\ref{fig:image-simplicial-complex}), together with the corresponding characterization of convexity.  In particular, Conjecture \ref{conj-Jeffs} holds for all cases from L1 to L23, and for minimal codes in the case of L24.   
    }  \label{tab:summary_of_results}
\end{table}

This article is structured as follows.
Section~\ref{sec:Background} contains the definitions and prior results necessary for understanding the main results in our work. Our main results are stated and proven in Section~\ref{sec:results}.
Finally, we give some concluding remarks in Section~\ref{sec:discussion}.

\section{Preliminaries} \label{sec:Background}

This section recalls definitions pertaining to neural codes (Section~\ref{sec:codes-background}), 
criteria for confirming or precluding convexity of neural codes (Sections~\ref{sec:yes-convex} and~\ref{sec:no-convex}), and results on the convexity of neural codes with few maximal codewords (Section~\ref{sec:few-max}).
We begin by recalling the definition of a simplicial complex (and its faces).

\begin{defn} \label{def:simplicial-complex}
A {\em simplicial complex} $\Delta$ on a nonempty finite set $V$ is a nonempty collection of subsets of~$V$ that is closed with respect to containment (that is, if $\sigma' \subseteq \sigma \in \Delta $, then $\sigma' \in \Delta$). Here, the set $V$ is the {\em vertex set} of $\Delta$.
If $\Delta$ is a simplicial complex, then 
\begin{enumerate}
    \item every $\sigma \in \Delta$ is a {\em face} of $\Delta$;
    \item if $\sigma$ is a face of $\Delta$, then $\sigma$ is a {\em $k$-simplex}, where $|\sigma| = k+1$;
    and 
    \item the faces of $\Delta$ that are maximal with respect to inclusion are the {\em facets} of $\Delta$.
\end{enumerate}
\end{defn}

The simplicial complexes we consider in this work are on two types of vertex sets: the set $[n] := \{1,2,\hdots, n\}$, where $n$ is a positive integer, and subsets of $\{F_1,F_2,F_3,F_4\}$, where the $F_i$ represent facets of another simplicial complex. Geometric realizations of all simplicial complexes (up to symmetry) on vertex sets of size $1$, $2$, $3$, or $4$ are shown in Figure~\ref{fig:image-simplicial-complex}.

\begin{figure}[htp]
    \centering
\begin{tabular}{|c|c|c|c|}
    \hline 
    \begin{subfigure}[t]{.2\textwidth}
        \centering
        \captionsetup{margin={-5pt,0pt}, slc=off, skip=-87.25pt}
        \begin{tikzpicture}
            \draw[draw=none, fill=none] (0,1) circle (1pt) node[anchor=south] {\vphantom{$F_8$}};
            \draw[draw=none, fill=none] (0,-1) circle (1pt) node[anchor=north] {\vphantom{$F_9$}};
            \filldraw[black] (0,-.25) circle (1pt) node[anchor=south]{\small $F_1$};
            \draw[draw=none] (0,.75)--(0,-1);
        \end{tikzpicture}
        \subcaption*{L1}
    \end{subfigure}&
    \begin{subfigure}[t]{.2\textwidth}
        \centering
        \captionsetup{margin={-3.5pt,0pt}, slc=off, skip=-87.25pt}
        \begin{tikzpicture}
            \draw[draw=none, fill=none] (0,1) circle (1pt) node[anchor=south] {\vphantom{$F_8$}};
            \draw[draw=none, fill=none] (0,-1) circle (1pt) node[anchor=north] {\vphantom{$F_9$}};
            \filldraw[black] (-.6,.2) circle (1pt) node[anchor=south]{\small $F_1$};
            \filldraw[black] (-1.35,-.55) circle (1pt) node[anchor=east]{\small $F_2$};
        \end{tikzpicture}
        \subcaption*{L2}
    \end{subfigure}&
    \begin{subfigure}[t]{.2\textwidth}
        \centering
        \captionsetup{margin={-3.5pt,0pt}, slc=off, skip=-87.25pt}
        \begin{tikzpicture}
            \draw[draw=none, fill=none] (0,1) circle (1pt) node[anchor=south] {\vphantom{$F_8$}};
            \draw[draw=none, fill=none] (0,-1) circle (1pt) node[anchor=north] {\vphantom{$F_9$}};
            \filldraw[black] (-.6,.2) circle (1pt) node[anchor=south]{\small $F_1$};
            \filldraw[black] (-1.35,-.55) circle (1pt) node[anchor=east]{\small $F_2$};
            \draw[thick, black] (-.6,.2)--(-1.35,-.55);
        \end{tikzpicture}
        \subcaption*{L3}
    \end{subfigure}&
    \begin{subfigure}[t]{.2\textwidth}
        \centering
        \captionsetup{margin={-3.5pt,0pt}, slc=off, skip=-87.25pt}
        \begin{tikzpicture}
            \draw[draw=none, fill=none] (0,1) circle (1pt) node[anchor=south] {\vphantom{$F_8$}};
            \draw[draw=none, fill=none] (0,-1) circle (1pt) node[anchor=north] {\vphantom{$F_9$}};
            
	        \filldraw[black] (0,.2) circle (1pt) node[anchor=south]{\small $F_1$};
			\filldraw[black] (.75,-.6) circle (1pt) node[anchor=west]{\small $F_3$};
			\filldraw[black] (-.75,-.6) circle (1pt) node[anchor=east]{\small $F_2$};
        \end{tikzpicture}
        \subcaption*{L4}
    \end{subfigure}\\
    \hline
    \begin{subfigure}[t]{.2\textwidth}
        \centering
        \captionsetup{margin={-3.5pt,0pt}, slc=off, skip=-87.25pt}
        \begin{tikzpicture}
            \draw[draw=none, fill=none] (0,.875) circle (1pt) node[anchor=south] {\vphantom{$F_8$}}; 
            \draw[draw=none, fill=none] (0,-1.125) circle (1pt) node[anchor=north] {\vphantom{$F_9$}}; 
            \filldraw[black] (0,.2) circle (1pt) node[anchor=south]{\small $F_1$};
			\filldraw[black] (.75,-.6) circle (1pt) node[anchor=west]{\small $F_3$};
			\filldraw[black] (-.75,-.6) circle (1pt) node[anchor=east]{\small $F_2$};
            \draw[thick, black] (0, .2)--(-.75,-.6);
        \end{tikzpicture}
        \subcaption*{L5}
    \end{subfigure}&
    \begin{subfigure}[t]{.2\textwidth}
        \centering
        \captionsetup{margin={-3.5pt,0pt}, slc=off, skip=-87.25pt}
        \begin{tikzpicture}
            \draw[draw=none, fill=none] (0,.875) circle (1pt) node[anchor=south] {\vphantom{$F_8$}}; 
            \draw[draw=none, fill=none] (0,-1.125) circle (1pt) node[anchor=north] {\vphantom{$F_9$}}; 
	        \filldraw[black] (0,.2) circle (1pt) node[anchor=south]{\small $F_1$};
			\filldraw[black] (.75,-.6) circle (1pt) node[anchor=west]{\small $F_3$};
			\filldraw[black] (-.75,-.6) circle (1pt) node[anchor=east]{\small $F_2$};
            \draw[thick, black] (0, .2)--(-.75,-.6);
            \draw[thick, black] (0,.2)--(.75,-.6);
        \end{tikzpicture}
        \subcaption*{L6}
    \end{subfigure}&
    \begin{subfigure}[t]{.2\textwidth}
        \centering
        \captionsetup{margin={-3.5pt,0pt}, slc=off, skip=-87.25pt}
        \begin{tikzpicture}
            \draw[draw=none, fill=none] (0,.875) circle (1pt) node[anchor=south] {\vphantom{$F_8$}}; 
            \draw[draw=none, fill=none] (0,-1.125) circle (1pt) node[anchor=north] {\vphantom{$F_9$}}; 
	        \filldraw[black] (0,.2) circle (1pt) node[anchor=south]{\small $F_1$};
			\filldraw[black] (.75,-.6) circle (1pt) node[anchor=west]{\small $F_3$};
			\filldraw[black] (-.75,-.6) circle (1pt) node[anchor=east]{\small $F_2$};
            \draw[thick, black] (0, .2)--(-.75,-.6);
            \draw[thick, black] (0,.2)--(.75,-.6);
            \draw[thick, black] (.75,-.6)--(-.75,-.6);
        \end{tikzpicture}
        \subcaption*{L7}
    \end{subfigure}&
    \begin{subfigure}[t]{.2\textwidth}
        \centering
        \captionsetup{margin={-3.5pt,0pt}, slc=off, skip=-87.25pt}
        \begin{tikzpicture}
            \draw[draw=none, fill=none] (0,.875) circle (1pt) node[anchor=south] {\vphantom{$F_8$}}; 
            \draw[draw=none, fill=none] (0,-1.125) circle (1pt) node[anchor=north] {\vphantom{$F_9$}}; 
	        \filldraw[black] (0,.2) circle (1pt) node[anchor=south]{\small $F_1$};
			\filldraw[black] (.75,-.6) circle (1pt) node[anchor=west]{\small $F_3$};
			\filldraw[black] (-.75,-.6) circle (1pt) node[anchor=east]{\small $F_2$};
            \draw [thick, draw=black, fill=gray, fill opacity=0.5] (0,.2) -- (.75,-.6) -- (-.75,-.6) -- cycle;
        \end{tikzpicture}
        \subcaption*{L8}
    \end{subfigure}\\
    \hline 
    \begin{subfigure}[t]{.2\textwidth}
        \centering
        \captionsetup{margin={-3.5pt,0pt}, slc=off, skip=-87.25pt}
        \begin{tikzpicture}
            \draw[draw=none, fill=none] (0,1.125) circle (1pt) node[anchor=south] {\vphantom{$F_8$}}; 
            \draw[draw=none, fill=none] (0,-.875) circle (1pt) node[anchor=north] {\vphantom{$F_9$}}; 
	        \filldraw[black] (0,.75) circle (1pt) node[anchor=south]{\small $F_1$};
			\filldraw[black] (.75,0) circle (1pt) node[anchor=west]{\small $F_3$};
			\filldraw[black] (0,-.75) circle (1pt) node[anchor=north]{\small $F_4$};
			\filldraw[black] (-.75,0) circle (1pt) node[anchor=east]{\small $F_2$};
            \draw[draw=none] (0,.75)--(0,-1);
        \end{tikzpicture}
        \subcaption*{L9}
    \end{subfigure}&
    \begin{subfigure}[t]{.2\textwidth}
        \centering
        \captionsetup{margin={-3.5pt,0pt}, slc=off, skip=-87.25pt}
        \begin{tikzpicture}
            \draw[draw=none, fill=none] (0,1.125) circle (1pt) node[anchor=south] {\vphantom{$F_8$}}; 
            \draw[draw=none, fill=none] (0,-.875) circle (1pt) node[anchor=north] {\vphantom{$F_9$}}; 
	        \filldraw[black] (0,.75) circle (1pt) node[anchor=south]{\small $F_1$};
			\filldraw[black] (.75,0) circle (1pt) node[anchor=west]{\small $F_3$};
			\filldraw[black] (0,-.75) circle (1pt) node[anchor=north]{\small $F_4$};
			\filldraw[black] (-.75,0) circle (1pt) node[anchor=east]{\small $F_2$};
            \draw[thick, black] (0, .75)--(-.75,0);
            \draw[draw=none] (0,.75)--(0,-1);
        \end{tikzpicture}
        \subcaption*{L10}
    \end{subfigure}&
    \begin{subfigure}[t]{.2\textwidth}
        \centering
        \captionsetup{margin={-3.5pt,0pt}, slc=off, skip=-87.25pt}
        \begin{tikzpicture}
            \draw[draw=none, fill=none] (0,1.125) circle (1pt) node[anchor=south] {\vphantom{$F_8$}}; 
            \draw[draw=none, fill=none] (0,-.875) circle (1pt) node[anchor=north] {\vphantom{$F_9$}}; 
	        \filldraw[black] (0,.75) circle (1pt) node[anchor=south]{\small $F_1$};
			\filldraw[black] (.75,0) circle (1pt) node[anchor=west]{\small $F_3$};
			\filldraw[black] (0,-.75) circle (1pt) node[anchor=north]{\small $F_4$};
			\filldraw[black] (-.75,0) circle (1pt) node[anchor=east]{\small $F_2$};
            \draw[thick, black] (0, .75)--(-.75,0);
            \draw[thick, black] (0, .75)--(.75,0);
            \draw[draw=none] (0,.75)--(0,-1);
        \end{tikzpicture}
        \subcaption*{L11}
    \end{subfigure}&
    \begin{subfigure}[t]{.2\textwidth}
        \centering
        \captionsetup{margin={-3.5pt,0pt}, slc=off, skip=-87.25pt}
        \begin{tikzpicture}
            \draw[draw=none, fill=none] (0,1.125) circle (1pt) node[anchor=south] {\vphantom{$F_8$}}; 
            \draw[draw=none, fill=none] (0,-.875) circle (1pt) node[anchor=north] {\vphantom{$F_9$}}; 
	        \filldraw[black] (0,.75) circle (1pt) node[anchor=south]{\small $F_1$};
			\filldraw[black] (.75,0) circle (1pt) node[anchor=west]{\small $F_3$};
			\filldraw[black] (0,-.75) circle (1pt) node[anchor=north]{\small $F_4$};
			\filldraw[black] (-.75,0) circle (1pt) node[anchor=east]{\small $F_2$};
            \draw[thick, black] (0, .75)--(-.75,0);
            \draw[thick, black] (0, -.75)--(.75,0);
            \draw[draw=none] (0,.75)--(0,-1);
        \end{tikzpicture}
        \subcaption*{L12}
    \end{subfigure}\\
    \hline
    \begin{subfigure}[t]{.2\textwidth}
        \centering
        \captionsetup{margin={-3.5pt,0pt}, slc=off, skip=-87.25pt}
        \begin{tikzpicture}
            \draw[draw=none, fill=none] (0,1.125) circle (1pt) node[anchor=south] {\vphantom{$F_8$}}; 
            \draw[draw=none, fill=none] (0,-.875) circle (1pt) node[anchor=north] {\vphantom{$F_9$}}; 
	        \filldraw[black] (0,.75) circle (1pt) node[anchor=south]{\small $F_1$};
			\filldraw[black] (.75,0) circle (1pt) node[anchor=west]{\small $F_3$};
			\filldraw[black] (0,-.75) circle (1pt) node[anchor=north]{\small $F_4$};
			\filldraw[black] (-.75,0) circle (1pt) node[anchor=east]{\small $F_2$};
            \draw[thick, black] (0, .75)--(-.75,0);
            \draw[thick, black] (0, -.75)--(.75,0);
            \draw[thick, black] (0, .75)--(.75,0);
            \draw[draw=none] (0,.75)--(0,-1);
        \end{tikzpicture}
        \subcaption*{L13}
    \end{subfigure}&
    \begin{subfigure}[t]{.2\textwidth}
        \centering
        \captionsetup{margin={-3.5pt,0pt}, slc=off, skip=-87.25pt}
        \begin{tikzpicture}
            \draw[draw=none, fill=none] (0,1.125) circle (1pt) node[anchor=south] {\vphantom{$F_8$}}; 
            \draw[draw=none, fill=none] (0,-.875) circle (1pt) node[anchor=north] {\vphantom{$F_9$}}; 
	        \filldraw[black] (0,.75) circle (1pt) node[anchor=south]{\small $F_1$};
			\filldraw[black] (.75,0) circle (1pt) node[anchor=west]{\small $F_3$};
			\filldraw[black] (0,-.75) circle (1pt) node[anchor=north]{\small $F_4$};
			\filldraw[black] (-.75,0) circle (1pt) node[anchor=east]{\small $F_2$};
            \draw[thick, black] (0, .75)--(-.75,0);
            \draw[thick, black] (0, -.75)--(0,.75);
            \draw[thick, black] (0, .75)--(.75,0);
            \draw[draw=none] (0,.75)--(0,-1);
        \end{tikzpicture}
        \subcaption*{L14}
    \end{subfigure}&
    \begin{subfigure}[t]{.2\textwidth}
        \centering
        \captionsetup{margin={-3.5pt,0pt}, slc=off, skip=-87.25pt}
        \begin{tikzpicture}
            \draw[draw=none, fill=none] (0,1.125) circle (1pt) node[anchor=south] {\vphantom{$F_8$}}; 
            \draw[draw=none, fill=none] (0,-.875) circle (1pt) node[anchor=north] {\vphantom{$F_9$}}; 
	        \filldraw[black] (0,.75) circle (1pt) node[anchor=south]{\small $F_1$};
			\filldraw[black] (.75,0) circle (1pt) node[anchor=west]{\small $F_3$};
			\filldraw[black] (0,-.75) circle (1pt) node[anchor=north]{\small $F_4$};
			\filldraw[black] (-.75,0) circle (1pt) node[anchor=east]{\small $F_2$};
            \draw[thick, black] (0, .75)--(-.75,0);
            \draw[thick, black] (.75,0)--(-.75, 0);
            \draw[thick, black] (0, .75)--(.75,0);
        \end{tikzpicture}
        \subcaption*{L15}
    \end{subfigure}&
    \begin{subfigure}[t]{.2\textwidth}
        \centering
        \captionsetup{margin={-3.5pt,0pt}, slc=off, skip=-87.25pt}
        \begin{tikzpicture}
            \draw[draw=none, fill=none] (0,1.125) circle (1pt) node[anchor=south] {\vphantom{$F_8$}}; 
            \draw[draw=none, fill=none] (0,-.875) circle (1pt) node[anchor=north] {\vphantom{$F_9$}}; 
	        \filldraw[black] (0,.75) circle (1pt) node[anchor=south]{\small $F_1$};
			\filldraw[black] (.75,0) circle (1pt) node[anchor=west]{\small $F_3$};
			\filldraw[black] (0,-.75) circle (1pt) node[anchor=north]{\small $F_4$};
			\filldraw[black] (-.75,0) circle (1pt) node[anchor=east]{\small $F_2$};
            \draw [thick, draw=black, fill=gray, fill opacity=0.5] (0,.75) -- (.75,0) -- (-.75,0) -- cycle;
        \end{tikzpicture}
        \subcaption*{L16}
    \end{subfigure}\\
    \hline
    \begin{subfigure}[t]{.2\textwidth}
        \centering
        \captionsetup{margin={-3.5pt,0pt}, slc=off, skip=-87.25pt}
        \begin{tikzpicture}
            \draw[draw=none, fill=none] (0,1.125) circle (1pt) node[anchor=south] {\vphantom{$F_8$}}; 
            \draw[draw=none, fill=none] (0,-.875) circle (1pt) node[anchor=north] {\vphantom{$F_9$}}; 
	        \filldraw[black] (0,.75) circle (1pt) node[anchor=south]{\small $F_1$};
			\filldraw[black] (.75,0) circle (1pt) node[anchor=west]{\small $F_3$};
			\filldraw[black] (0,-.75) circle (1pt) node[anchor=north]{\small $F_4$};
			\filldraw[black] (-.75,0) circle (1pt) node[anchor=east]{\small $F_2$};
            \draw[thick, black] (0, .75)--(-.75,0);
            \draw[thick, black] (.75,0)--(-.75, 0);
            \draw[thick, black] (0, .75)--(.75,0);
            \draw[thick, black] (-.75, 0)--(0,-.75);
        \end{tikzpicture}
        \subcaption*{L17}
    \end{subfigure}&
    \begin{subfigure}[t]{.2\textwidth}
        \centering
        \captionsetup{margin={-3.5pt,0pt}, slc=off, skip=-87.25pt}
        \begin{tikzpicture}
            \draw[draw=none, fill=none] (0,1.125) circle (1pt) node[anchor=south] {\vphantom{$F_8$}}; 
            \draw[draw=none, fill=none] (0,-.875) circle (1pt) node[anchor=north] {\vphantom{$F_9$}}; 
	        \filldraw[black] (0,.75) circle (1pt) node[anchor=south]{\small $F_1$};
			\filldraw[black] (.75,0) circle (1pt) node[anchor=west]{\small $F_3$};
			\filldraw[black] (0,-.75) circle (1pt) node[anchor=north]{\small $F_4$};
			\filldraw[black] (-.75,0) circle (1pt) node[anchor=east]{\small $F_2$};
            \draw [thick, draw=black, fill=gray, fill opacity=0.5] (0,.75) -- (.75,0) -- (-.75,0) -- cycle;
            \draw[thick, black] (-.75, 0)--(0,-.75);
        \end{tikzpicture}
        \subcaption*{L18}
    \end{subfigure}&
    \begin{subfigure}[t]{.2\textwidth}
        \centering
        \captionsetup{margin={-3.5pt,0pt}, slc=off, skip=-87.25pt}
        \begin{tikzpicture}
            \draw[draw=none, fill=none] (0,1.125) circle (1pt) node[anchor=south] {\vphantom{$F_8$}}; 
            \draw[draw=none, fill=none] (0,-.875) circle (1pt) node[anchor=north] {\vphantom{$F_9$}}; 
	        \filldraw[black] (0,.75) circle (1pt) node[anchor=south]{\small $F_1$};
			\filldraw[black] (.75,0) circle (1pt) node[anchor=west]{\small $F_3$};
			\filldraw[black] (0,-.75) circle (1pt) node[anchor=north]{\small $F_4$};
			\filldraw[black] (-.75,0) circle (1pt) node[anchor=east]{\small $F_2$};
            \draw[thick, black] (0, .75)--(-.75,0);
            \draw[thick, black] (0, .75)--(.75,0);
            \draw[thick, black] (-.75, 0)--(0,-.75);
            \draw[thick, black] (0,-.75) -- (.75,0);
        \end{tikzpicture}
        \subcaption*{L19}
    \end{subfigure}&
    \begin{subfigure}[t]{.2\textwidth}
        \centering
        \captionsetup{margin={-3.5pt,0pt}, slc=off, skip=-87.25pt}
        \begin{tikzpicture}
            \draw[draw=none, fill=none] (0,1.125) circle (1pt) node[anchor=south] {\vphantom{$F_8$}}; 
            \draw[draw=none, fill=none] (0,-.875) circle (1pt) node[anchor=north] {\vphantom{$F_9$}}; 
	        \filldraw[black] (0,.75) circle (1pt) node[anchor=south]{\small $F_1$};
			\filldraw[black] (.75,0) circle (1pt) node[anchor=west]{\small $F_3$};
			\filldraw[black] (0,-.75) circle (1pt) node[anchor=north]{\small $F_4$};
			\filldraw[black] (-.75,0) circle (1pt) node[anchor=east]{\small $F_2$};
            \draw[thick, black] (0, .75)--(-.75,0);
            \draw[thick, black] (0, .75)--(.75,0);
            \draw[thick, black] (0,-.75) -- (.75,0);
            \draw[dashed, thick, black] (.75,0)--(-.75,0);
            \draw[thick, black] (0,.75)--(0,-.75);
        \end{tikzpicture}
        \subcaption*{L20}
    \end{subfigure}\\
    \hline
    \begin{subfigure}[t]{.2\textwidth}
        \centering
        \captionsetup{margin={-3.5pt,0pt}, slc=off, skip=-87.25pt}
        \begin{tikzpicture}
            \draw[draw=none, fill=none] (0,1.125) circle (1pt) node[anchor=south] {\vphantom{$F_8$}}; 
            \draw[draw=none, fill=none] (0,-.875) circle (1pt) node[anchor=north] {\vphantom{$F_9$}}; 
	        \filldraw[black] (0,.75) circle (1pt) node[anchor=south]{\small $F_1$};
			\filldraw[black] (.75,0) circle (1pt) node[anchor=west]{\small $F_3$};
			\filldraw[black] (0,-.75) circle (1pt) node[anchor=north]{\small $F_4$};
			\filldraw[black] (-.75,0) circle (1pt) node[anchor=east]{\small $F_2$};
            \draw [thick, draw=black, fill=gray, fill opacity=0.5] (0,.75) -- (.75,0) -- (-.75,0) -- cycle;
            \draw[thick, black] (-.75, 0)--(0,-.75);
            \draw[thick, black] (.75,0)--(0,-.75);
        \end{tikzpicture}
        \subcaption*{L21}
    \end{subfigure}&
    \begin{subfigure}[t]{.2\textwidth}
        \centering
        \captionsetup{margin={-3.5pt,0pt}, slc=off, skip=-87.25pt}
        \begin{tikzpicture}
            \draw[draw=none, fill=none] (0,1.125) circle (1pt) node[anchor=south] {\vphantom{$F_8$}}; 
            \draw[draw=none, fill=none] (0,-.875) circle (1pt) node[anchor=north] {\vphantom{$F_9$}}; 
            \filldraw[black] (0,.75) circle (1pt) node[anchor=south]{\small $F_1$};
            \filldraw[black] (.75,0) circle (1pt) node[anchor=west]{\small $F_3$};
            \filldraw[black] (0,-.75) circle (1pt) node[anchor=north]{\small $F_4$};
			\filldraw[black] (-.75,0) circle (1pt) node[anchor=east]{\small $F_2$};
            \draw [thick, draw=black, fill=gray, fill opacity=0.5] (0,.75) -- (.75,0) -- (-.75,0) -- cycle;
            \draw [thick, draw=black, fill=gray, fill opacity=0.5] (0,-.75)--(.75,0)--(-.75,0)--cycle;    
        \end{tikzpicture}
        \subcaption*{L22}
    \end{subfigure}&
    \begin{subfigure}[t]{.2\textwidth}
        \centering
        \captionsetup{margin={-3.5pt,0pt}, slc=off, skip=-87.25pt}
        \begin{tikzpicture}
            \draw[draw=none, fill=none] (0,1.125) circle (1pt) node[anchor=south] {\vphantom{$F_8$}}; 
            \draw[draw=none, fill=none] (0,-.875) circle (1pt) node[anchor=north] {\vphantom{$F_9$}}; 
	        \filldraw[black] (0,.75) circle (1pt) node[anchor=south]{\small $F_1$};
			\filldraw[black] (.75,0) circle (1pt) node[anchor=west]{\small $F_3$};
			\filldraw[black] (0,-.75) circle (1pt) node[anchor=north]{\small $F_4$};
			\filldraw[black] (-.75,0) circle (1pt) node[anchor=east]{\small $F_2$};
            \draw[thick, black] (0, .75)--(-.75,0);
            \draw[thick, black] (0, .75)--(.75,0);
            \draw[thick, black] (0,-.75) -- (.75,0);
            \draw[thick, black, dashed] (.75,0)--(-.75,0);
            \draw[thick, black] (0,.75)--(0,-.75);
            \draw[thick, black] (0,-.75) -- (-.75,0);
        \end{tikzpicture}
        \subcaption*{L23}
    \end{subfigure}&
    \begin{subfigure}[t]{.2\textwidth}
        \centering
        \captionsetup{margin={-3.5pt,0pt}, slc=off, skip=-87.25pt}
        \begin{tikzpicture}
            \draw[draw=none, fill=none] (0,1.125) circle (1pt) node[anchor=south] {\vphantom{$F_8$}}; 
            \draw[draw=none, fill=none] (0,-.875) circle (1pt) node[anchor=north] {\vphantom{$F_9$}}; 
	        \filldraw[black] (0,.75) circle (1pt) node[anchor=south]{\small $F_1$};
			\filldraw[black] (.75,0) circle (1pt) node[anchor=west]{\small $F_3$};
			\filldraw[black] (0,-.75) circle (1pt) node[anchor=north]{\small $F_4$};
			\filldraw[black] (-.75,0) circle (1pt) node[anchor=east]{\small $F_2$};
            \draw [thick, dashed, draw=black, fill=gray, fill opacity=0.5] (0,.75) -- (.75,0) -- (-.75,0) -- cycle;
            \draw[thick, black] (-.75,0)--(0,.75)--(.75,0);
            \draw[thick, black] (-.75, 0)--(0,-.75);
            \draw[thick, black] (.75,0)--(0,-.75);
            \draw[thick, black] (0,.75)--(0,-.75);
        \end{tikzpicture}
        \subcaption*{L24}
    \end{subfigure} \\
    \hline
    \begin{subfigure}[t]{.2\textwidth}
        \centering
        \captionsetup{margin={-3.5pt,0pt}, slc=off, skip=-87.25pt}
        \begin{tikzpicture}
            \draw[draw=none, fill=none] (0,1.125) circle (1pt) node[anchor=south] {\vphantom{$F_8$}}; 
            \draw[draw=none, fill=none] (0,-.875) circle (1pt) node[anchor=north] {\vphantom{$F_9$}}; 
	        \filldraw[black] (0,.75) circle (1pt) node[anchor=south]{\small $F_1$};
			\filldraw[black] (.75,0) circle (1pt) node[anchor=west]{\small $F_3$};
			\filldraw[black] (0,-.75) circle (1pt) node[anchor=north]{\small $F_4$};
			\filldraw[black] (-.75,0) circle (1pt) node[anchor=east]{\small $F_2$};
            \draw [thick, dashed, draw=black, fill=gray, fill opacity=0.5] (0,.75) -- (.75,0) -- (-.75,0) -- cycle;
            \draw [thick, draw=black, fill=gray, fill opacity=0.9] (0,.75) -- (.75,0) -- (0, -.75) -- cycle;
            \draw[thick, black] (-.75,0)--(0,.75)--(.75,0);
            \draw[thick, black] (-.75, 0)--(0,-.75);
            \draw[thick, black] (.75,0)--(0,-.75);
            \draw[thick, black] (0,.75)--(0,-.75);
        \end{tikzpicture}
        \subcaption*{L25}
    \end{subfigure}&
    \begin{subfigure}[t]{.2\textwidth}
        \centering
        \captionsetup{margin={-3.5pt,0pt}, slc=off, skip=-87.25pt}
        \begin{tikzpicture}
            \draw[draw=none, fill=none] (0,1.125) circle (1pt) node[anchor=south] {\vphantom{$F_8$}}; 
            \draw[draw=none, fill=none] (0,-.875) circle (1pt) node[anchor=north] {\vphantom{$F_9$}}; 
	        \filldraw[black] (0,.75) circle (1pt) node[anchor=south]{\small $F_1$};
			\filldraw[black] (.75,0) circle (1pt) node[anchor=west]{\small $F_3$};
			\filldraw[black] (0,-.75) circle (1pt) node[anchor=north]{\small $F_4$};
			\filldraw[black] (-.75,0) circle (1pt) node[anchor=east]{\small $F_2$};
            \draw[thick, black] (0,.75)--(.75,0)--(0,-.75)--cycle;
            \draw[thick, black] (0,.75)--(-.75,0)--(0,-.75);
            \draw[very thin, dashed] (.75,0)--(-.75,0);
            \fill [fill=gray, fill opacity=0.65] (0,.75) -- (.75,0) -- (-.75,0) -- cycle;
            \fill [ fill=gray, fill opacity=0.8] (0,.75) -- (.75,0) -- (0, -.75) -- cycle;
            \fill [fill=gray, fill opacity=0.35] (0,-.75)--(-.75,0)--(.75,0)--cycle;
            \draw[thick, black] (0,.75)--(.75,0)--(0,-.75)--cycle;
        \end{tikzpicture}
        \subcaption*{L26}
    \end{subfigure}&
    \begin{subfigure}[t]{.2\textwidth}
        \centering
        \captionsetup{margin={-3.5pt,0pt}, slc=off, skip=-87.25pt}
        \begin{tikzpicture}
            \draw[draw=none, fill=none] (0,1.125) circle (1pt) node[anchor=south] {\vphantom{$F_8$}}; 
            \draw[draw=none, fill=none] (0,-.875) circle (1pt) node[anchor=north] {\vphantom{$F_9$}}; 
	        \filldraw[black] (0,.75) circle (1pt) node[anchor=south]{\small $F_1$};
			\filldraw[black] (.75,0) circle (1pt) node[anchor=west]{\small $F_3$};
			\filldraw[black] (0,-.75) circle (1pt) node[anchor=north]{\small $F_4$};
			\filldraw[black] (-.75,0) circle (1pt) node[anchor=east]{\small $F_2$};
            \draw [thick, dashed, draw=black, fill=gray, fill opacity=0.5] (0,.75) -- (.75,0) -- (-.75,0) -- cycle;
            \draw[thick, black] (-.75,0)--(0,.75)--(.75,0);
            \draw [thick, draw=black, fill=gray, fill opacity=0.5] (0,.75) -- (.75,0) -- (0, -.75) -- cycle;
            \draw[thick, draw=black, fill=gray, fill opacity=0.5] (0,.75)--(-.75,0)--(0,-.75)--cycle;
            \draw[thick, draw=black, dashed, fill=gray, fill opacity=0.5] (.75,0)--(-.75,0)--(0,-.75)--cycle;
            \draw[thick, black] (-.75,0)--(0,-.75)--(.75,0);
        \end{tikzpicture}
        \subcaption*{L27}
    \end{subfigure}&
    \begin{subfigure}[t]{.2\textwidth}
        \centering
        \captionsetup{margin={-3.5pt,0pt}, slc=off, skip=-87.25pt}
        \begin{tikzpicture}
            \draw[draw=none, fill=none] (0,1.125) circle (1pt) node[anchor=south] {\vphantom{$F_8$}}; 
            \draw[draw=none, fill=none] (0,-.875) circle (1pt) node[anchor=north] {\vphantom{$F_9$}}; 
	        \filldraw[black] (0,.75) circle (1pt) node[anchor=south]{\small $F_1$};
			\filldraw[black] (.75,0) circle (1pt) node[anchor=west]{\small $F_3$};
			\filldraw[black] (0,-.75) circle (1pt) node[anchor=north]{\small $F_4$};
			\filldraw[black] (-.75,0) circle (1pt) node[anchor=east]{\small $F_2$};
            \draw [thick, draw=black, fill=black, fill opacity=.75] (0,.75) -- (.75,0) -- (0, -.75) -- cycle;
            \draw[thick, draw=black, fill=black, fill opacity=.75] (0,.75)--(-.75,0)--(0,-.75)--cycle;
            \draw[dashed, very thin] (.75,0)--(-.75,0);
        \end{tikzpicture}
        \subcaption*{L28}
    \end{subfigure} \\
    \hline
    \end{tabular}
    \caption{All simplicial complexes on up to $4$ vertices, up to symmetry. 
    The simplicial complexes are labeled L1 to L28, matching the labels in~\cite{Curto} (in fact, this figure aligns closely with~\cite[Figure~5]{Curto}).
    }
    \label{fig:image-simplicial-complex}
\end{figure}
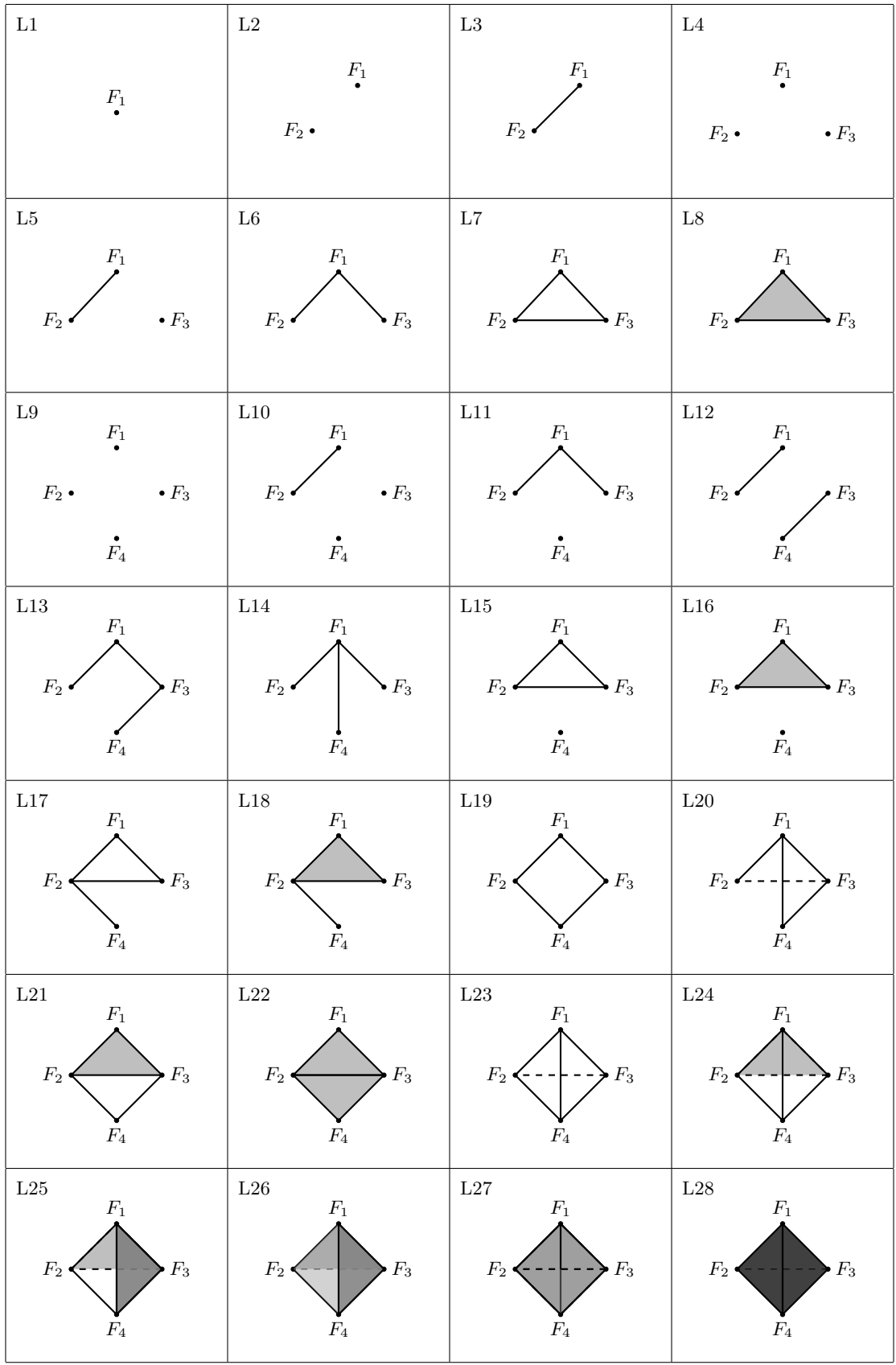

\subsection{Neural codes} \label{sec:codes-background}

\begin{defn}
A \textit{neural code} (or, for short, \textit{code}) $\C$ on $n$ neurons is a collection of subsets $c  \subseteq {[n]}$, called \textit{codewords}. A codeword $c \in \C$ is \textit{maximal} if $c$ is maximal with respect to set inclusion.  The \textit{simplicial complex}, $\Delta(\C)$, of a code $\C$ on $n$ neurons is defined by $\Delta(\C) \coloneqq \{\sigma \subseteq [n] : \sigma \subseteq c\textrm{~for~some~} c \in \C\}$. An \textit{$n$-maximal} code is a code with exactly $n$ maximal codewords. 
\end{defn}
\noindent
By construction, the maximal codewords of a code $\C$ are precisely the facets of the simplicial complex $\Delta(\C)$.

\noindent \textbf{Notation 2.3.} To simplify notation, we write each codeword as a string of integers, rather than in the roster form of a set. For instance, $\{1,2,3,4\}$ is written as $1234$. Additionally, in our examples we use boldface to indicate maximal codewords.

This section has three running examples, which are codes we call $\C_{22}$, $\C_{24}$, and $\C_{26}$. The significance of these labels will be clear at the end of the section (in Examples~\ref{ex:nerve-is-L22}--\ref{ex:nerve-is-L26}).

\begin{ex}[$\C_{22}$] \label{ex:first-example-L22}
The following is a $4$-maximal code on $7$ neurons: 
\begin{align}
    \label{eq:C22}
\C_{22} ~:=~ \{ 
\mathbf{134}, \mathbf{1357}, \mathbf{257},\mathbf{356},
13, 35, 57, \varnothing\}~.
\end{align}
The (geometric realization of the) simplicial complex $\Delta(\C_{22})$ is shown here:
   
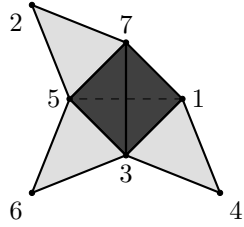
\begin{center}
                  \begin{tikzpicture}
	   \filldraw[black] (0,.75) circle (1pt) node[anchor=south] {$7$};
			\filldraw[black] (.75,0) circle (1pt) node[anchor=west] {$1$};
			\filldraw[black] (-.75,0) circle (1pt) node[anchor=east]{$5$};
			\filldraw[black] (0,-.75) circle (1pt) node[anchor=north] {$3$};
			\filldraw[black] (1.25,-1.25) circle (1pt) node[anchor=north west] {$4$};
			\filldraw[black] (-1.25,1.25) circle (1pt) node[anchor=north east] {$2$};
			\filldraw[black] (-1.25,-1.25) circle (1pt) node[anchor=north east] {$6$};
            \draw [thick, draw=black, fill=black, fill opacity=.75] (0,.75) -- (.75,0) -- (0, -.75) -- cycle;
            \draw[thick, draw=black, fill=black, fill opacity=.75] (0,.75)--(-.75,0)--(0,-.75)--cycle;
            \draw[thick, draw=black, fill=gray, fill opacity=.25] (0,-.75)--(.75,0)--(1.25,-1.25)--cycle;
            \draw[thick, draw=black, fill=gray, fill opacity=.25] (0,.75)--(-.75,0)--(-1.25,1.25)--cycle;
            \draw[thick, draw=black, fill=gray, fill opacity=.25] (0,-.75)--(-.75,0)--(-1.25,-1.25)--cycle;
            \draw[dashed, very thin] (.75,0)--(-.75,0);
        \end{tikzpicture}
   \end{center}
The four facets of  $\Delta(\C_{22})$ correspond to the four maximal codewords of $\C_{22}$; these facets are the $3$-simplex (filled-in tetrahedron) $1357$ and the three $2$-simplices (filled-in triangles) $134$, $356$, and $237$.
\end{ex}

Codes can be generated by collections of subsets of Euclidean space, as follows.
\begin{defn} \label{def:realization}
Given a collection $\mathcal{U} = ({U}_i)_{i=1}^n $ of subsets of $\mathbb{R}^d$, for some $d \geq 1$, the \textit{neural code generated by $\mathcal{U}$} is
\begin{equation*}
    \mathcal{C}(\mathcal{U}) ~\coloneqq~ \left\{\sigma \subseteq [n]: \big(\textstyle\bigcap_{i \in \sigma} U_i \big) \smallsetminus \big(\textstyle\bigcup_{j \not\in \sigma} U_j \big) \ne \varnothing\right\}~,
\end{equation*}

\noindent and $\mathcal{U}$ is a {\em realization} of $\mathcal{C}(\mathcal{U})$.  
If $\C = \C(\mathcal{U})$ for a collection  $\mathcal{U} = \{U_i\}_{i=1}^n$ of sets that are convex and open, the code $\C$ is \textit{open convex} (or, for simplicity, \textit{convex}). 
\end{defn}

\begin{assumption}[Empty codeword]
\label{assumption:empty-set}
We will assume that every code contains the empty set as a codeword; this assumption does not affect the convexity of the code \cite[Remark 2.5]{decidability}.  
\end{assumption}

\begin{remark} \label{rem:open-vs-closed}
Some prior works consider realizations of convex sets that are \uline{closed} (see, for instance,~\cite{nondegen,Giusti}), in addition to those that are open.
In our work, however, we restrict our attention to open (convex) sets.
\end{remark}

\begin{remark}[Labels in realizations] 
In our depictions of convex realizations $\mathcal{U} = (U_i)_{i=1}^{n}$, we label each region 
$\big(\textstyle\bigcap_{i \in \sigma} U_i \big) \smallsetminus \big(\textstyle\bigcup_{j \not\in \sigma} U_j \big)$
by the corresponding codeword, $\sigma \in \C(\mathcal{U})$.  (The exception is for the empty codeword $\sigma = \varnothing$; we do not label the corresponding region which is simply the complement of the union of all other regions.)
Accordingly, for any neuron $i$, the corresponding subset $U_i$ is the union of the regions labeled by codewords containing $i$ (more precisely, $U_i$ is the interior of the union of the closures of all regions $\big(\textstyle\bigcap_{i \in \sigma} U_i \big) \smallsetminus \big(\textstyle\bigcup_{j \not\in \sigma} U_j \big)$ for which $i \in \sigma$).  For instance, in Figure~\ref{fig:runningex}, the set $U_1$ is the open rectangle formed by the regions labeled by $134$, $13$, and $1357$.
\end{remark}

\begin{figure}[ht]
    \centering
    \begin{tikzpicture}[scale =2.5]
    \draw (0,0) rectangle (4,1);

    \draw (1,0) -- (1,1);
    \draw (3,-1) -- (3,1);
    \draw (4,0) -- (4,1);
    \draw (5,1)--(3,-1);
    \draw (4,1)--(5,1);
    \draw (4.5, 1)--(4.5,.5);
    \draw (3,-.5)--(3.5,-.5);

    \node at (.5,.5) {$\mathbf{134}$};
    \node at (2,.5) {$13$};
    \node at (3.5,.5) {$\mathbf{1357}$};
    \node at (4.25,.65) {$35$};
    \node at (3.15,-.65) {$\mathbf{257}$};
    \node at (3.4, -.2) {$57$};
    \node at (4.65, .85) {$\mathbf{356}$};
\end{tikzpicture}
    \caption{Convex realization of $\C_{22} = \{
\mathbf{134}, \mathbf{1357},  \mathbf{257},\mathbf{356},
13, 35, 57, \varnothing\}$.}
    \label{fig:runningex}
\end{figure}

\begin{ex}[Example~\ref{ex:first-example-L22} continued]\label{ex: running example L22}
Figure~\ref{fig:runningex} displays a convex realization of the code 
$$\C_{22} ~=~ \{
\mathbf{134}, \mathbf{1357},  \mathbf{257},\mathbf{356},
13, 35, 57, \varnothing\}~.$$ 
Thus, $\C_{22}$ is convex.
\end{ex}

\begin{remark} \label{rem:connect-literature}
    The convex realization in Figure~\ref{fig:runningex} is planar (that is, the sets $U_i$ are open subsets of $\mathbb{R}^2$).  In fact, all realizations appearing in this article are planar (or even one-dimensional).
    An algorithm that determines whether a given convex code has a planar convex realization was given recently by Bukh and Jeffs~\cite{planarcodes-decidable}.  Additionally, many of our realizations use (open) axis-parallel boxes like the realizations in \cite{axis-parallel} and codes of a certain degree in \cite{embedding-dim-gaps}.
\end{remark}

Detecting the convexity of a given neural code is NP-hard~\cite{matroids}. Nevertheless, for some families of codes, there are simple criteria for assessing convexity.  
In what follows, we recall such criteria for confirming convexity (Section~\ref{sec:yes-convex})
and for 
precluding convexity (Section~\ref{sec:no-convex}).

\subsection{Criteria for confirming convexity} \label{sec:yes-convex}
In this subsection, we recall two results (Lemmas~\ref{lem:monotonicity} and~\ref{Lem:max-intersection}).  The first  is a monotonicity result of Cruz {\em et al.}, which states that convexity is maintained when non-maximal codewords are added to a code~\cite[Theorem~1.3]{Cruz}, as follows.

\begin{lemma}[Monotonicity of convexity]\label{lem:monotonicity}
Let $\mathcal{C}, \mathcal{D}$ be codes such that $\mathcal{C} \subseteq \mathcal{D} \subseteq \Delta(\mathcal{C})$. Then if $\mathcal{C}$ is convex, $\mathcal{D}$ is convex. 
\end{lemma}

The next criterion involves the inclusion of ``max-intersection'' faces, as defined below.

\begin{defn} \,
\begin{enumerate}[(a)]
    \item Let $\Delta$ be a simplicial complex.  A face $\sigma \in \Delta$ is a \emph{max-intersection face} of  $\Delta$ if $\sigma$ is nonempty and is the intersection of two or more facets of $\Delta$.
    \item Let $\C$ be a code with simplicial complex $\Delta(\C)$.
    A \emph{max-intersection face} of $\mathcal{C}$ is a max-intersection face of $\simp{\C}$. Next, $\C$ is \emph{max-intersection-complete} if it contains every max-intersection face of $\C$ (or, equivalently, if $\C$ contains every intersection of two or more codewords of $\C$).
\end{enumerate}
\end{defn}
The following result is due to Cruz {\em et al.}~\cite[Theorem~1.2]{Cruz}.

\begin{lemma}[Max-intersection-complete codes are convex]\label{Lem:max-intersection}
If $\mathcal{C}$ is a code that is max-intersection-complete, then $\mathcal{C}$ is convex.
\end{lemma}

The converse of Lemma~\ref{Lem:max-intersection} is false, as the following example shows.
\begin{ex}[Example~\ref{ex: running example L22} continued]\label{ex:max-intersection-L22}
Recall that the code 
$\C_{22} = \{ 
\mathbf{134}, \mathbf{1357}, \mathbf{257},\mathbf{356},
13, 35, 57, \varnothing\}$ 
is convex.  
We label the facets of $\Delta(\C_{22})$ by $F_1=134$, $F_2=1357$, $F_3=356$, and $F_4=257$ (the ordering of these facets is for later convenience when comparing with Figure~\ref{fig:image-simplicial-complex}).  
Among the intersections of two or more of these facets, the nonempty ones are as follows:
\begin{equation} \label{eq:facet-intersection-L22}
    F_1\cap F_2 =13,~
    \quad 
    F_1\cap F_2 \cap F_3=3,~
    \quad
F_2\cap F_3 =35,~
    \quad 
    F_2\cap F_3 \cap F_4=5,~
    \quad 
F_2\cap F_4 =57.
\end{equation}
Neither $3$ nor $5$ is a codeword of $\C_{22}$, so $\C_{22}$ is not max-intersection-complete.
\end{ex}

\subsection{Criteria for precluding convexity} \label{sec:no-convex}
In this subsection, we recall two results that can be used to show that a code is not convex (Lemmas~\ref{lem:sprockets-nonconvex} and~\ref{lem: convex-no-obs}).

\subsubsection{Wheels and sprockets} \label{sec-wheels}
In this subsection, we recall the definitions of a wheel and a sprocket~\cite{wheels}.
Sprockets generate wheels, and the presence of a wheel is an obstruction to convexity (Lemma~\ref{lem:sprockets-nonconvex} below). 

\begin{defn}[Wheel] \label{def:wheel}
    Let $\C$ be a code on $n$ neurons, and let $\mathcal{U} = \{U_i\}_{i=1}^n$ be a realization of~$\C$. A tuple $\mathcal{W} = (\sigma_1, \sigma_2, \sigma_3, \tau) \in (2^{[n]})^4$ is a \textit{wheel of the realization} $\mathcal{U}$ if it satisfies:
        \begin{enumerate}[start=1,label={\bfseries{W(\roman*)}:}]   
        \item $U_{\sigma_i} \cap U_{\sigma_j} = U_{\sigma_1} \cap U_{\sigma_2} \cap U_{\sigma_3} \neq \varnothing$ for all $1 \leq j < k \leq 3$,
        
        \item $U_{\sigma_1} \cap U_{\sigma_2} \cap U_{\sigma_3} \cap U_\tau = \varnothing$,
        
        \item if $U_\tau$ and $U_{\sigma_j} \cap U_\tau$ are convex for $j=1,2,3,$ then there exists a line segment such that one of the endpoints lies in $U_{\sigma_1} \cap U_\tau$ and the other in $U_{\sigma_3} \cap U_\tau$, that also meets $U_{\sigma_2} \cap U_\tau$.  
    \end{enumerate}
    Finally, $\mathcal{W}$ is a \textit{wheel of} $\C$ if it is a wheel of every realization $\mathcal{U}$ of $\C$.
\end{defn}

Wheels are defined relative to realizations of a code $\C$, but there are ``combinatorial wheels'' -- combinatorial conditions that can be read from $\C$ and $\Delta(\C)$ -- that guarantee the presence of a wheel.  One such combinatorial wheel is a sprocket (Definition~\ref{def:wheel-sprocket} below). To define sprockets, we need Jeffs' definition of a ``trunk''~\cite{morphisms}.

\begin{defn}[Trunk]
Let $\mathcal{C}$ be a code on $n$ neurons, and let $\sigma \subseteq [n]$. The \textit{trunk} of $\sigma$ in $\mathcal{C}$ is \[\text{Tk}_{\mathcal{C}}(\sigma) ~:=~ \{c \in \mathcal{C} : \sigma \subseteq c\}~.\]
\end{defn}

\begin{defn}[Partial-wheel and sprocket] \label{def:wheel-sprocket}
Let $\C$ be a neural code.
    A tuple $\mcW = (\sigma_1, \sigma_2, \sigma_3, \tau) \in (\Delta(\mcC))^4$ is a \emph{partial-wheel} of $\mcC$ if it satisfies the following conditions:
    \begin{enumerate}[start=1,label={\bfseries{P(\roman*)}:}]   
        \item $\sigma_1 \cup \sigma_2 \cup \sigma_3 \in \Delta(\mcC)$, and $\Tk_{\mcC}(\sigma_j \cup \sigma_k)= \Tk_{\mcC}(\sigma_1 \cup \sigma_2 \cup \sigma_3)$ for every $1 \leq j < k \leq 3$,
        
        \item $\sigma_1 \cup \sigma_2 \cup \sigma_3 \cup \tau \not\in \Delta(\mcC)$, and 
        
        \item[\textbf{P(iii)$_\circ$:}] $\sigma_j \cup \tau \in \Delta(\mcC)$ for $j \in \{1,2,3\}$. 
    \end{enumerate}
     A \textit{sprocket} of a $\mcC$ is a partial-wheel $\mcW = (\sigma_1, \sigma_2, \sigma_3, \tau)$ of $\mcC$ for which 
     there exist $\rho_1$, $\rho_3 \in \Delta(\mcC)$ (these sets $\rho_1$ and $\rho_3$ are called \textit{witnesses} for the spocket) such that the following conditions hold:
        \begin{enumerate}[start=1,label={\bfseries{S}(\arabic*):}, itemindent=5pt]
            \item $\Tk_{\mcC}(\sigma_j \cup \tau)\subseteq \Tk_{\mcC}(\rho_j)$ for $j \in \{1,3\}$,
            
            \item $\Tk_{\mcC}(\tau) \subseteq \Tk_{\mcC}(\rho_1) \cup \Tk_{\mcC}(\rho_3)$, and 
            
            \item $\Tk_{\mcC}(\rho_1 \cup \rho_3 \cup \tau) \subseteq \Tk_{\mcC}(\sigma_2)$.
            
        \end{enumerate}           
\end{defn}

\begin{remark}\label{rem:notation-sprocket}
Our notation {\bf P(i)--P(iii)$_\circ$} matches that in~\cite{wheels}, whereas for simplicity we use {\bf S(1)--S(3)} rather than the notation {\bf S(iii)(1)--S(iii)(3)} used in~\cite{wheels}.
\end{remark}

The following result is due to Ruys de Perez {\em et al.}~\cite[Theorem~3.4 and Proposition~4.5]{wheels}.

\begin{lemma}[Sprockets are obstructions to convexity] \label{lem:sprockets-nonconvex}
    Let $\C$ be a code.
    \begin{enumerate}
        \item If $\C$ has a sprocket, then $\C$ has a wheel.
        \item If $\C$ has a wheel, then $\C$ is non-convex.
    \end{enumerate}
\end{lemma}

\begin{ex}[A code with a sprocket] \label{ex: sprocket-L24}
Consider the following $4$-maximal code: $$\C_{24} ~:=~ \{\mathbf{123},\mathbf{1246},\mathbf{145}, \mathbf{356}, 12, 14, 3 ,5 ,6, \varnothing\}~.$$ 
We claim that $\mcW = (3,6,5,1)$, with witnesses $\rho_1 = 12$ and $\rho_3 = 14$, is a sprocket of $\C_{24}$, and therefore (by Lemma~\ref{lem:sprockets-nonconvex}) $\C_{24}$ is non-convex. We briefly outline the proof of this claim here:
\begin{itemize}
    \item \textbf{P(i)}: $356\in \C_{24}$ implies that $356\in \Delta(\C_{24})$, and $\Tk_{\C_{24}}(35) = \Tk_{\C_{24}}(36) = \Tk_{\C_{24}}(56) =\{356\} = \Tk_{\C_{24}}(356)$.
    \item \textbf{P(ii)}: None of the codewords in ${\C_{24}}$ contain $1356$, so $1356\not\in \Delta(\C_{24})$.
    \item \textbf{P(iii)}$_\circ$: $13 \subseteq 123 \in \C_{24}$, $15\subseteq 145\in \C_{24}$, and $16\subseteq 1246\in \C_{24}$; hence, $13, 15,16\in \Delta(\C_{24})$. 
    \item \textbf{S(1)}: $\Tk_{\C_{24}}(13) = \{123\} \subseteq \{12,123\} = \Tk_{\C_{24}}(12)$, and $\Tk_{\C_{24}}(15) = \{145\} \subseteq \{14,145\} = \Tk_{\C_{24}}(14)$.
    \item \textbf{S(2)}: $\Tk_{\C_{24}}(1) = \{123, 1246, 145, 12, 14\} = \{123, 1246, 12\} \cup \{1246, 145, 14\} = \Tk_{\C_{24}}(12) \cup \Tk_{\C_{24}}(14)$.
    \item \textbf{S(3)}: $\Tk_{\C_{24}}(124) = \{1246\} \subseteq \{1246,356,6\} = \Tk_{\C_{24}}(6)$.
\end{itemize}
\end{ex}

\begin{remark} \label{rem:converse-sprocket}
The converse of 
    Lemma~\ref{lem:sprockets-nonconvex}(1) is false~\cite[Example~4.7]{wheels}.
\end{remark}

\subsubsection{Local obstructions to convexity} \label{sec:local-obstruction}

This subsection recalls ``local obstructions'' to convexity (Definition~\ref{def:local-obs}).  

\begin{defn}[Link] \label{def:link}
Let $\Delta$ be a simplicial complex on $[n]$, and let $\sigma \in \Delta$. The \emph{link of} $\sigma$ \emph{in} $\Delta$ is the following simplicial complex:
\begin{equation*}
    \link{\Delta}{\sigma} ~\coloneqq ~\{\tau \subseteq [n]: \tau \cap \sigma = \varnothing \text{ and } \tau \cup \sigma \in \Delta\}~.
\end{equation*}
\end{defn}

The following definition of local obstructions is equivalent to the usual one~\cite[Proposition~4.8]{Curto}.

\begin{defn}[Local obstruction] \label{def:local-obs}
    ~
    \begin{enumerate}
    \item Let $\Delta$ be a simplicial complex. A face $\sigma \in \Delta$ is \textit{mandatory} if $\sigma \neq \varnothing$ and the link $\mathrm{Lk}_{\Delta}(\sigma)$ is not contractible.
    \item Let $\C$ be a code with simplicial complex $\Delta(\C)$. A {\em mandatory face} of $\C$ is a mandatory face of $\Delta(\C)$.  Next, $\C$ has a {\em local obstruction} if there exists a mandatory face $\sigma$ of $\C$ such that $\sigma \notin \C$.
    \end{enumerate}
\end{defn}

It follows from definitions that a code $\C$ has no local obstructions if and only if  $\mathcal{C}_{\text{min}}(\Delta(\mathcal{C})) \subseteq \C$, where the \textit{minimal code} $\mathcal{C}_{\text{min}}(\Delta)$ of a 
simplicial complex $\Delta$ is given by
    \begin{align} \label{eq:c-min} 
    \notag
    \mathcal{C_{\text{min}}}(\Delta) ~&:=~  \{ \sigma \in \Delta : \mathrm{Lk}_{\Delta}(\sigma) \mathrm{~is~not~contractible} \} \cup \{ \varnothing\} 
    \\
    ~&=~
         \{\text{facets of $\Delta$} \} \cup 
    \{\text{non-facet mandatory faces of $\Delta$}\}
    \cup \{\varnothing\} ~.
    \end{align}

The following result is due to Giusti and Itskov~\cite[Theorem~3]{Giusti} (see also~\cite[Lemma~1.2]{Curto}).

\begin{lemma}[Local obstruction to convexity]
\label{lem: convex-no-obs} 
    If $\C$ is a code with a local obstruction, then $\C$ is non-convex.
\end{lemma}

Next, we recall that mandatory faces are necessarily intersections of facets~\cite[Lemma~1.4]{Curto}.

\begin{lemma} \label{lem:intersection-facet}
    If $\sigma$ is a mandatory face of a simplicial complex $\Delta$, then $\sigma$ either is a facet of $\Delta$ or is the intersection of two or more facets of $\Delta$. 
\end{lemma}

\begin{ex}[Example~\ref{ex: sprocket-L24} continued]\label{ex:local-obstruction-L24}
We use Lemma~\ref{lem:intersection-facet} to investigate the mandatory faces of the code
$\C_{24} = \{\mathbf{123},\mathbf{1246},\mathbf{145}, \mathbf{356}, 12, 14, 3 ,5 ,6, \varnothing\}$.
We consider intersections of two or more of the facets $F_1=123$, $F_2=1246$, $F_3=145$, and $F_4=356$ of $\Delta(\C_{24})$; the nonempty such intersections are as follows:
\begin{equation} \label{eq:facet-intersections-L24}
        F_1 \cap F_2 \cap F_3 = F_1 \cap F_3 = 1,~
    F_1 \cap F_2 =12,~
    F_2 \cap F_3 =14,~
    F_1 \cap F_4 =3,~
    F_2 \cap F_4 =6,~
    F_3 \cap F_4 =5.    
\end{equation}
All of these intersections are codewords of $\C_{24}$, except $\sigma=1$.  The corresponding link, $\mathrm{Lk}_{\Delta(\C_{24})}(1)$, is shown below and is readily seen to be contractible:

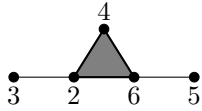
\begin{center}
\begin{tikzpicture}[scale=.8]
	\draw (0,0) --(1,0);
	\draw (2,0) --(3,0);
    \draw [fill=gray, thick] (1,0) --(2,0)--(1.5,.8) -- (1,0);
    \draw [fill] (0,0) circle [radius=0.08];
    \draw [fill] (1,0) circle [radius=0.08];
    \draw [fill] (2,0) circle [radius=0.08];
    \draw [fill] (3,0) circle [radius=0.08];
    \draw [fill] (1.5,.8) circle [radius=0.08];
    \node [below] at (0,0) {$3$};
    \node [below] at (1,0) {$2$};
    \node [below] at (2,0) {$6$};
    \node [above] at (1.5,.8) {$4$};
    \node [below] at (3,0) {$5$};
	\end{tikzpicture}
\end{center}
We conclude that $\C_{24}$ has no local obstructions.  Next, it is straightforward to check that the links of all remaining $5$ intersections shown in~\eqref{eq:facet-intersections-L24} -- that is, $\mathrm{Lk}_{\Delta(\C_{24})}(\sigma)$ for $\sigma  \in \{12, 14, 3, 6, 5\}$ -- are non-contractible (in fact, disconnected).  Therefore, these $5$ faces $\sigma$ are mandatory, and thus $\C_{24}$ is in fact the minimal code of $\Delta(\C_{24})$ (that is, $\mincode(\Delta(\C_{24}))=\C_{24}$).
\end{ex}
Example~\ref{ex:local-obstruction-L24} shows that the converse of Lemma~\ref{lem: convex-no-obs} does not hold.

Next, we recall what is known about the converse of Lemma~\ref{lem:intersection-facet} in the cases of the intersections of two or three facets: The nonempty intersection of two (and not more) facets is always mandatory (Lemma~\ref{lem:intersection-2-facets} below),
while for three facets, being mandatory is characterized by violating a ``Path-of-Facets Condition'' (Definition~\ref{def:path-facets} and Lemma~\ref{lem:triple}(a) below).  The following result is~\cite[Lemma~4.7]{Curto}.

\begin{lemma}[Intersection of $2$ facets]\label{lem:intersection-2-facets}
Let $\Delta$ be a simplicial complex. If $\sigma = F_1 \cap F_2$, for distinct facets $F_1$ and $F_2$ of $\Delta$, and $\sigma$ is not contained in any other facet of $\Delta$, then $\mathrm{Lk}_{\Delta}(\sigma)$ is not contractible. 
\end{lemma}

\begin{defn} \label{def:path-facets}
Let $\Delta$ be a simplicial complex with exactly three facets $F_1$, $F_2$, and $F_3$. Then $\Delta$ satisfies the \emph{Path-of-Facets Condition} if exactly one of the following three sets is empty: 
\begin{itemize}
    \item $(F_1 \cap F_2) \smallsetminus F_3$
    \item $(F_1 \cap F_3) \smallsetminus F_2$
    \item $(F_2 \cap F_3) \smallsetminus F_1$.
\end{itemize}  
\end{defn}

The following result is~\cite[Lemmas~3.1 and~3.3]{Shiu}, with additional details coming from~\cite[proof of Lemma~3.3]{Shiu}.

\begin{lemma}[Simplicial complexes with $3$ facets]\label{lem:triple}\label{lem:path}
Let $\Delta$ be a simplicial complex with exactly three facets $F_1$, $F_2$, and $F_3$.
\begin{enumerate}[(a)]
    \item   The link $\link{\Delta}{F_1 \cap F_2 \cap F_3}$ is contractible if and only if $\Delta$ satisfies the Path-of-Facets Condition.
    \item  If $\Delta$ satisfies the Path-of-Facets Condition, then the minimal code $\mathcal{C}_{\text{min}}(\Delta)$ is convex.  To give more details, if 
    $(a,b,c)$ is a permutation of $(1,2,3)$  
     such that $(F_a \cap F_b) \smallsetminus F_c$ and $(F_b \cap F_c)\smallsetminus F_a$ are nonempty, while $(F_a \cap F_c) \smallsetminus F_b$ is empty; then $\mathcal{C}_{\text{min}}(\Delta) = \{ \mathbf{F_a}, \mathbf{F_b}, \mathbf{F_c}, F_a \cap F_b, F_b \cap F_c\}$, and a convex realization of $\mathcal{C}_{\text{min}}(\Delta)$ in $\mathbb{R}$ is shown in Figure~\ref{fig:realization-1-dim}.
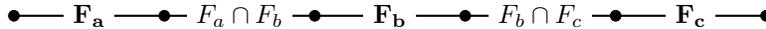
\begin{figure}[ht]
    \centering        
     \begin{tikzpicture}
        \filldraw[black] (2,0) circle (2pt);
        \filldraw[black] (0,0) circle (2pt);
        \filldraw[black] (-2,0) circle (2pt);
        \filldraw[black] (-4,0) circle (2pt);
        \filldraw[black] (-6,0) circle (2pt);
        \filldraw[black] (4,0) circle (2pt);

        \draw[thick, black] (-6,0)--(-4,0) node [midway, fill=white] {$\mathbf{F_a}$};
        \draw[thick, black] (-4,0)--(-2,0) node [midway, fill=white] {$F_a \cap F_b$};
        \draw[thick, black] (-2,0)--(0,0) node [midway, fill=white] {$\mathbf{F_b}$};
        \draw[thick, black] (0,0)--(2,0) node [midway, fill=white] {$F_b \cap F_c$};
        \draw[thick, black] (2,0)--(4,0) node [midway, fill=white] {$\mathbf{F_c}$};
    \end{tikzpicture}
\caption{A convex realization of 
$\mathcal{C}_{\text{min}}(\Delta)$, where $\Delta$ is a simplicial complex with exactly three facets and satisfies the Path-of-Facets Condition.}
\label{fig:realization-1-dim}
\end{figure}%
\end{enumerate}
\end{lemma}

\begin{ex}[Example~\ref{ex:max-intersection-L22} continued] \label{ex:path-facets-L22}
We revisit the code
$\C_{22} = \{ 
\mathbf{134}, \mathbf{1357}, \mathbf{257},\mathbf{356},
13, 35, 57, \varnothing\}$, and consider the simplicial complex $\Delta$ which has three of the four facets of $\Delta(\C_{22})$, specifically, $F_1=134$, $F_2=1357$, and $F_3=356$.  (In other words, $\Delta = \Delta(\{F_1,F_2,F_3\})$.)  
The simplicial complex $\Delta$ satisfies the Path-of-Facets Condition, because $(F_1 \cap F_2) \smallsetminus F_3 = 1$,  $(F_1 \cap F_3) \smallsetminus F_2=5$, and $(F_2 \cap F_3) \smallsetminus F_1 = \varnothing$.  Thus, by Lemma~\ref{lem:path}(b), $\mathcal{C}_{\text{min}}(\Delta) =
    \{ \mathbf{F_1}, \mathbf{F_2}, \mathbf{F_3}, F_1 \cap F_2, F_2 \cap F_3\} = \{ \mathbf{134}, \mathbf{1357}, \mathbf{356}, 13, 35\}$,  
    and a convex realization of $\mathcal{C}_{\text{min}}(\Delta)$, which has been ``extended'' into $\mathbb{R}^2$, is shown here: 
\begin{center}
    \begin{tikzpicture}[scale = 1.5]
    \draw (-1.5,0) rectangle (8,1);

    \draw (.5,0) -- (.5,1);
    \draw (2.5,0) -- (2.5,1);
    \draw (4,0) -- (4,1);
    \draw (6,0) -- (6,1);

    \node at (-.5,.5) {$\mathbf{F_1=134}$};
    \node at (1.5,.5) {$F_1\cap F_2=13$};
    \node at (3.25,.5) {$\mathbf{F_2=1357}$};
    \node at (5,.5) {$F_2 \cap F_3=35$};
    \node at (7.,.5) {$\mathbf{F_3=356}$};
\end{tikzpicture}
\end{center}
Observe that this realization of $\mathcal{C}_{\text{min}}(\Delta)$ can be modified to obtain the realization (of the larger code~$\C_{22}$) shown earlier in Figure~\ref{fig:runningex}.  This idea underlies a later proof (in Section~\ref{sec:L22}).
\end{ex}

\begin{ex}[Example~\ref{ex:local-obstruction-L24} continued]\label{ex:path-facets-L24}
Recall that the 
facets of $\Delta(\C_{24})$ are $F_1=123$, $F_2=1246$, $F_3=145$, and $F_4=356$.   
It follows that 
$(F_1 \cap F_3) \smallsetminus F_2 = \varnothing$, while  $(F_1 \cap F_2) \smallsetminus F_3= 2$ and 
$(F_2 \cap F_3) \smallsetminus F_1 = 4$.  Therefore, $\Delta(\C_{24})$ satisfies the Path-of-Facets Condition. 
\end{ex}

\subsection{Codes with few maximal codewords} \label{sec:few-max}
Johnston {\em et al.}\ used Lemma~\ref{lem:path} to show that, for codes with up to three maximal codewords, convexity is characterized by having no local obstructions.  That is, for such codes, the converse of Lemma~\ref{lem: convex-no-obs} holds, as follows~\cite[Theorem~1.1]{Shiu}.

\begin{prop} \label{prop:3-max}
    If $\C$ is a code with up to three maximal codewords, then $\C$ is convex if and only if $\C$ has no local obstructions.
\end{prop}

Proposition~\ref{prop:3-max} does not extend to codes with four maximal codewords (i.e., 4-maximal codes).  
Indeed, we saw (in Examples~\ref{ex: sprocket-L24} and~\ref{ex:local-obstruction-L24}) that the $4$-maximal code 
$\C_{24}$ has no local obstructions and yet is non-convex (due to having a sprocket).  A similar such example is as follows.

\begin{ex} \label{ex:original-counterexample-L26}
Consider the following 4-maximal code: 
\begin{align}
    \label{eq:C26}
\mathcal{C}_{26} ~:=~ 
    \{\mathbf{123}, \mathbf{134}, \mathbf{145}, 
    \mathbf{2345}, 
    13, 14, 23, 34, 45, 4, 5, \varnothing\}~.
\end{align}
This was the first code found to be non-convex despite having no local obstructions~\cite[Theorem~3.1]{Lienkaemper}. The non-convexity comes from having a wheel (coming from a ``wheel frame'')~\cite[Proposition~4.16]{wheels}.  
\end{ex}

In light of the above examples, it is natural to ask whether these two types of obstructions to convexity -- wheels and local obstructions -- are enough to characterize convexity in 4-maximal codes.  Jeffs conjectures that the answer is ``yes'' (Conjecture~\ref{conj-Jeffs}),
and this conjecture is known to hold for all codes on up to $5$ neurons~\cite[Theorem~3.8]{wheels} and all \uline{minimal} codes on $6$ neurons~\cite[Table~1]{wheels}.

\begin{remark}
In light of Conjecture~\ref{conj-Jeffs}, it is natural to ask whether -- for codes with {\em more than four} maximal codewords -- wheels and local obstructions together characterize convexity.  This question is open.  
A starting point for such an investigation might be a set of 96 5-maximal codes on 6 neurons that are minimal codes (and thus have no local obstructions), but whose convexity status is unknown (also unknown  is whether they contain wheels)~\cite[Table~1]{wheels}.
\end{remark}

The goal of our work is to resolve many cases of Conjecture~\ref{conj-Jeffs}.
Our strategy is to categorize 4-maximal codes based on 
the structure of their maximal codewords and their intersections.  In other words, we consider the ``nerve'' of the set of maximal codewords, which we define next.

\begin{defn}[Nerve $\mathcal{N}(\mathcal{F}(\C))$] \label{def:nerve}
 Let $\C$ be a code, and let $\mathcal{F}(\C)=\{F_1,F_2,\dots,F_m\}$ denote the set of maximal codewords of $\C$ (equivalently, $\mathcal{F}(\C)$ is the set of facets of $\Delta(\C)$).  
    The \textit{nerve of the set of maximal codewords} of $\C$, which we denote by $\mathcal{N}(\mathcal{F(C)})$, 
    is the simplicial complex on $\mathcal{F}(\C)$ in which $\{F_{i_1},F_{i_2},\dots, F_{i_k}\}$ is a face if and only if the intersection $F_{i_1}\cap F_{i_2} \cap \dots \cap F_{i_k}$ is nonempty.
\end{defn}

In what follows, we label codes based on their nerve $\mathcal{N}(\mathcal{F(C)})$.  For instance, an ``L22 code'' is one for which $\mathcal{N}(\mathcal{F(C)})$ is the simplicial complex labeled by L22 in Figure~\ref{fig:image-simplicial-complex}.  See the following notation.
    
\begin{notation}
    An ``L$\#\#$ code'' is a code $\C$ for which the nerve $\mathcal{N}(\mathcal{F(C)})$ is the simplicial complex $L\#\#$ in Figure \ref{fig:image-simplicial-complex}.
\end{notation}

\begin{ex}[Example~\ref{ex:path-facets-L22} continued; L22 code] \label{ex:nerve-is-L22}
For the code $\C_{22} = \{\mathbf{134},\mathbf{1357},\mathbf{257},\mathbf{356}, 13, 35, 57, \varnothing\}$, recall that we label the maximal codewords (equivalently, facets of $\Delta(\C_{22})$) by 
$F_1=134$, $F_2=1357$, $F_3=356$, and $F_4=257$.  Intersections of these maximal codewords were shown earlier in~\eqref{eq:facet-intersection-L22}, and this computation allows us to conclude that the nerve of the set of maximal codewords of $\C_{22}$ consists of two $2$-simplices $\{F_1,F_2,F_3\}$ and $\{F_2,F_3,F_4\}$ (and all of their subsets).  We conclude that this nerve is the simplicial complex labeled by L22 in Figure~\ref{fig:image-simplicial-complex}.
\end{ex}

\begin{ex}[Example~\ref{ex:path-facets-L24} continued; L24 code] \label{ex:nerve-is-L24}
We return to the following code: 
$$\C_{24} ~=~ \{\mathbf{123}, \mathbf{1246}, \mathbf{145}, \mathbf{356}, 12, 14, 3 ,5 ,6, \varnothing\}~.$$
From the intersections of facets~\eqref{eq:facet-intersections-L24}, we see that the nerve $\mathcal{N}(\mathcal{F}(\C_{24}))$ consists of a $2$-simplex $\{F_1,F_2,F_3\}$ and three edges $\{F_1,F_4\}$, $\{F_2,F_4\}$, and $\{F_3,F_4\}$.  In other words, the nerve is simplicial complex labeled by L24 in Figure~\ref{fig:image-simplicial-complex}.
\end{ex}

\begin{ex}[Example~\ref{ex:original-counterexample-L26} continued; L26 code] \label{ex:nerve-is-L26}
Recall that the code $\C_{26}$, from~\eqref{eq:C26},
has $4$ maximal codewords, which we label as follows:
$F_1=2345$, $F_2=123$, $F_3=134$, $F_4=145$.  Observe that $F_1 \cap F_2\cap F_4 = \varnothing$, while all other triplewise intersections $F_i \cap F_j \cap F_k$ are nonempty.  We conclude that the  
    nerve $\mathcal{N}(\mathcal{F}(\C_{26}))$ is the simplicial complex labeled by L26 in Figure~\ref{fig:image-simplicial-complex}.
\end{ex}

We end this section by recalling a result that can be used to confirm the contractibility of links.  The following result was proven using the Nerve Lemma~\cite[Equation~(2)]{Lienkaemper}.

\begin{lemma} \label{nerve}
Let $\sigma \in \Delta$, where $\Delta$ is a simplicial complex. Denote the set of facets of the link $\mathrm{Lk}_{\Delta}(\sigma)$ by 
    \[\mathcal{L}_{\Delta}(\sigma) ~=~ \{F \smallsetminus \sigma : F \,\, \text{is a facet of} \,\, \Delta \,\, \text{containing} \,\,  \sigma\}~.\] Then we have a homotopy equivalence $\mathrm{Lk}_{\Delta}(\sigma) \simeq \mathcal{N}(\mathcal{L}_{\Delta}(\sigma))$. In particular, $\mathrm{Lk}_{\Delta}(\sigma)$ is contractible if and only if $\mathcal{N}(\mathcal{L}_{\Delta}(\sigma))$ is contractible. 
\end{lemma}

\section{Results} \label{sec:results}

In this section, we present our results on the convexity of $4$-maximal codes.  
As explained in the prior section, our strategy is to consider separate cases based on the nerve of the set of maximal codewords.  Indeed, such a nerve $\mathcal{N}$ is a simplicial complex on $4$ vertices and therefore is (up to symmetry) one of the simplicial complexes shown earlier in Figure~\ref{fig:image-simplicial-complex}. More precisely, $\mathcal{N}$ is one of the simplicial complexes L9 to L28. 
Our main result concerns the cases of L9 to L23, as follows.

\begin{thm}[Main result, L9--L23 codes] \label{thm:summary}
    Let $\C$ be a code with four maximal codewords, 
    and let $\mathcal{N}$ denote the nerve of the set of maximal codewords of $\C$. 
If $\mathcal{N}$ is one of the simplicial complexes L9 to L23 (in Figure~\ref{fig:image-simplicial-complex}), then $\C$ is convex if and only if $\C$ has no local obstructions.
\end{thm}

Theorem~\ref{thm:summary} immediately resolves many cases of Conjecture~\ref{conj-Jeffs}. 
\begin{cor} \label{cor:conjecture}
    Conjecture~\ref{conj-Jeffs} is true for all codes $\C$ for which the nerve of the set of 
    maximal codewords 
    is one of the simplicial complexes L9 to L23 (in Figure~\ref{fig:image-simplicial-complex}).
\end{cor}

In the case of the simplicial complex L24, the above result partially extends, as follows.

\begin{thm}[L24 codes] \label{thm:L24-summary}
~
\begin{enumerate}
    \item  Conjecture~\ref{conj-Jeffs} is true for all \uline{minimal} codes $\C$ for which the nerve of the set of 
    maximal codewords 
    is the simplicial complex L24.  
    \item There exists a ($4$-maximal) L24 code that 
has no local obstructions, but 
has a sprocket and therefore is non-convex.
\end{enumerate}
\end{thm}

Theorem~\ref{thm:L24-summary}(2) implies that, unlike in the case of  Theorem~\ref{thm:summary}, convexity of L24 codes 
requires more than avoiding local obstructions. Additionally, Theorem \ref{thm:L24-summary}(2) has already been proven; indeed, such a code and its relevant properties were shown in Examples~\ref{ex: sprocket-L24} and~\ref{ex:local-obstruction-L24}.
Accordingly, the rest of this section is dedicated to proving Theorems~\ref{thm:summary} and~\ref{thm:L24-summary}(1).  Specifically, Theorem~\ref{thm:summary} is obtained by combining Proposition~\ref{prop:disconnected-nerve}, Corollary~\ref{cor:no-triangle}, and Theorems \ref{thm: l18-l21} and~\ref{thm: l22} below; while Theorem~\ref{thm:L24-summary}(1) follows from Theorem~\ref{thm: sprocket-L24}.

\subsection{Nerves that are disconnected or lack $2$-simplices } \label{sec:discon-or-no-2-simplices}

 For codes whose maximal-codeword nerve is disconnected, the connected components of the nerve can be treated separately and  so the classification of convexity for codes with up to $3$ maximal codewords applies, as follows.
 \begin{prop}[L9--L12, L15--L16] \label{prop:disconnected-nerve}
    If $\C$ is a $4$-maximal code and the nerve of the set of maximal codewords is disconnected  -- equivalently, this nerve is, up to relabeling, one of the simplicial complexes L9 to L12, L15, or L16 in Figure~\ref{fig:image-simplicial-complex} -- then $\C$ is convex if and only if $\C$ has no local obstructions.
 \end{prop}
 \begin{proof} Convex codes have no local obstructions (Lemma~\ref{lem: convex-no-obs}), so we need only prove the converse.  Accordingly, assume that $\C$ is $4$-maximal and has no local obstructions, and also that the nerve of the set of maximal codewords of $\C$, which we denote by $\mathcal{N},$ is disconnected.  It follows that we can partition $\mathcal N$ into two or more blocks $\mathcal{N}_1,\mathcal{N}_2,\dots,\mathcal{N}_{\ell}$, so that maximal codewords in distinct blocks do not have any neurons in common.  This implies that we can express $\C$ as a union of codes, 
 $\C = \{\varnothing\} \cup \C_1 \cup \C_2 \cup \dots \cup \C_{\ell}$, where $\C_i$ consists of the nonempty codewords of $\C$ that are contained in at least one maximal codeword $\sigma \in \mathcal{N}_i$.  By construction, the $\C_i$'s are disjoint.  
 
 Each of the codes $\C_i \cup \{\varnothing\}$ has at most $3$ maximal codewords and therefore (by Proposition~\ref{prop:3-max}) has a convex realization.  Now a convex realization of $\C$ is obtained by taking the ``union'' of these realizations (more precisely, by taking copies of these realizations, possibly ``extending'' some into higher dimensions, and placing them next to each other in some $\mathbb{R}^d$).
 \end{proof}

Our next result pertains to codes with any number of maximal codewords.  Specifically, the result considers the case of nerves that consist only of vertices and edges, in other words, nerves with no $2$-simplices. 

\begin{prop}[No $2$-simplices]\label{prop:notriangle}
Let $\C$ be a code.  If the nerve of the set of maximal codewords of $\C$ contains no $2$-simplices, then the following are equivalent:
    \begin{enumerate}
        \item $\C$ is convex, 
        \item $\C$ has no local obstructions, and
        \item $\C$ is max-intersection-complete.
    \end{enumerate}
\end{prop}

\begin{proof}
Let $\mathcal{F}(\C) = \{F_1,F_2,\hdots, F_m\}$ denote the set of maximal codewords of $\mathcal{C}$.  Our assumption that the nerve of $\mathcal{F}(\C)$ has no $2$-simplices, simply means that $F_i \cap F_j \cap F_k = \varnothing$, whenever $1 \leq i<j<k \leq m$.

 The implications $(3) \Rightarrow (1)$ and $(1) \Rightarrow (2)$ are
 Lemmas \ref{Lem:max-intersection} and \ref{lem: convex-no-obs}, respectively. Thus, it suffices to prove the implication $(2) \Rightarrow (3)$. 

 We proceed by contradiction.
 Suppose that $\mathcal{C}$ is not max-intersection-complete. Then there exists $\sigma \not\in \mathcal{C}$ such that $\sigma$ is a max-intersection face of $\C$. Thus, by definition, $\sigma$ is the intersection of two or more maximal codewords.  Recall, however, that every triple-wise intersection of maximal codewords is empty.  We conclude that $\sigma = F_i \cap F_j$ for some $1 \leq i<j \leq m$, and $\sigma \not\subseteq F_k$ for 
 all $k \in [m] \smallsetminus \{i,j\}$. 
 Thus, Lemma \ref{lem:intersection-2-facets} implies that $\link{\simp{\C}}{\sigma}$ is not contractible. By assumption, $\sigma \not\in \mathcal{C}$, so $\sigma$ is a mandatory face of $\C$, and therefore $\C$ has a local obstruction (by Definition~\ref{def:local-obs}(2)).
\end{proof}

Proposition~\ref{prop:notriangle} yields the following corollary.

\begin{cor} \label{cor:no-triangle}
    If $\C$ is a $4$-maximal code and the nerve of the set of maximal codewords is one of the following simplicial complexes:
    L9, L13, L14, L17, L19, L20, or L23, then $\C$ is convex if and only if $\C$ has no local obstructions.
\end{cor}

\begin{remark}
    L9 codes are covered by both Proposition~\ref{prop:disconnected-nerve} and Corollary~\ref{cor:no-triangle}.
\end{remark}

\subsection{Preliminary results on nerves with one 2-simplex}
In the prior subsection, we analyzed codes for which the nerve has no $2$-simplices.  We now turn our attention to nerves with exactly one $2$-simplex.  The relevant nerves are L18, L21, and L24. (L8 and L16 also have only one $2$-simplex, but L8 is on only three vertices and L16 is disconnected, so Propositions~\ref{prop:3-max} and~\ref{prop:disconnected-nerve} apply).  

\begin{lemma}\label{lem:restrict-facet}
    Let $\C$ be a $4$-maximal neural code, and let
    $\mathcal{F}(\C)=\{F_1,F_2,F_3,F_4\}$ denote the set of maximal codewords.
    If $\C$ has no local obstructions and the nerve $\mathcal{N}(\mathcal{F}(\C))$ is L18, L21, or L24 (in particular, $F_1 \cap F_2 \cap F_3$ is the unique nonempty triplewise-intersection of maximal codewords),
        then both of the following sets are subsets of $\C$:
\[
\mincode^\ast ~: =~ \mincode(\simp{\{F_1,F_2,F_3\}})
\quad {\rm and}
\quad 
\C_{\cap F_4}
~:=~
\{F_4\} \cup 
\{F_1 \cap F_4\} \cup \{ F_2 \cap F_4
\} \cup \{F_3 \cap F_4\}~.
\]
\end{lemma}

\begin{proof}
We first show that $\mincode^\ast  \subseteq \C$.  We begin by writing the following, using~\eqref{eq:c-min}:
\begin{align*}
    \mincode^\ast ~=~ \{F_1,F_2,F_3\} \cup \{ \text{non-facet mandatory faces of }\C^*\} \cup \{\varnothing\}~,
\end{align*}
where $\C^*:=\{F_1,F_2,F_3\}$.  By hypothesis and Assumption~\ref{assumption:empty-set}, $\{\varnothing, F_1, F_2, F_3\} \subseteq \C$.  Therefore, we need only consider non-facet mandatory faces $\gamma$ of $\C^*$.  Accordingly, assume that $\gamma$ is such a face; in other words (by Lemma~\ref{lem:intersection-facet}), $\varnothing \neq \gamma = F_i \cap F_j \cap F_k$ for some $i,j,k \in \{1,2,3\}$ with $i \neq j$ (but $i=k$ is possible) such that $\link{\simp{\C^*}}{\gamma}$ is not contractible.
Our aim is to show that $\gamma \in \C$.

We claim that $\gamma \nsubseteq F_4$. Indeed, this follows from (1) the fact that $\gamma$ is the intersection of two or more of the facets $F_1,F_2,F_3$; and (2)~$F_1\cap F_2 \cap F_3$ is the only nonempty intersection of three facets of $\Delta(\C)$.  We conclude that the set of facets of $\Delta(\C)$ that contain~$\gamma$ is the same as the set of facets of $\Delta(\C^*)$ that contain~$\gamma$. 
Hence, using the notation in Lemma \ref{nerve}, we have $ \mathcal{L}_{\simp{\C}}(\gamma)= \mathcal{L}_{\simp{\C^*}}(\gamma)$.  Thus, Lemma~\ref{nerve} yields the following:
\[
\link{\simp{\C}}{\gamma} ~\simeq~
\mathcal{N}(\mathcal{L}_{\simp{\C}}(\gamma)) ~= ~
\mathcal{N}(\mathcal{L}_{\simp{\C^*}}(\gamma))
~\simeq~
\link{\simp{\C^*}}{\gamma}~.
\]
We conclude that $\link{\simp{\C}}{\gamma}$ is not contractible
(because 
$\link{\simp{\C^*}}{\gamma}$ is not contractible).  Therefore, $\gamma$ is a mandatory face of $\C$ and so (since $\C$ has no local obstructions) $\gamma \in \C$, as desired.

Now we show that $\C_{\cap F_4} \subseteq \C$. Let $i \in \{1,2,3\}$. If $F_i \cap F_4 = \varnothing$, then $F_i \cap F_4 \in \C$ (recall Assumption~\ref{assumption:empty-set}). Now assume $F_i \cap F_4 \neq \varnothing$. As $\mathcal{N}(\mathcal{F}(\C))$ is L18, L21, or L24, it follows that for $k\in\{1,2,3\} \smallsetminus \{i\}$, we have $F_i \cap F_4 \not\subseteq F_k$. Now it follows from Lemma \ref{lem:intersection-2-facets} that $\link{\simp{\C}}{F_i \cap F_4}$ is not contractible, and therefore $F_i \cap F_4$ is a mandatory face of $\C$. Since $\C$ has no local obstructions, we have $F_i \cap F_4 \in \C$. Finally, 
$F_4$ is a maximal codeword of $\C$, so $F_4 \ cin \C$. We conclude that $\C_{\cap F_4} \subseteq \C$.
\end{proof}

\begin{prop} \label{prop: pathoffaces-convex}
    Let $\C$ be a $4$-maximal neural code, and let
    $\mathcal{F}(\C)=\{F_1,F_2,F_3,F_4\}$ denote the set of maximal codewords.
    If the following hold:
    \begin{enumerate}
        \item $\C$ has no local obstructions,
        \item the nerve $\mathcal{N}(\mathcal{F}(\C))$ is L18, L21, or L24 (in particular, $F_1 \cap F_2 \cap F_3$ is the unique nonempty triplewise-intersection of maximal codewords), and
        \item $\Delta(\{F_1,F_2,F_3\})$ does \uline{not} satisfy the Path-of-Facets Condition;
    \end{enumerate}
    then $\C$ is max-intersection-complete and thus convex. 
\end{prop}

\begin{proof}
Define $\C^\ast = \{F_1,F_2,F_3\}$.
    By Lemma \ref{lem:restrict-facet}, both $\mincode^\ast = \mincode(\Delta(\C^\ast))$ and 
    $\C_{\cap F_4} = \{F_1 \cap F_4, F_2 \cap F_4, F_3 \cap F_4, F_4\}$ are subsets of the code $\C$. 
    
    We must show that every nonempty max-intersection face $\gamma$ of $\C$ is a codeword of $\C$.  
    Each such face $\gamma$ must come from the following list, because the nerve $\mathcal{N}(\mathcal{F}(\C))$ is L18, L21, or L24 (in Figure~\ref{fig:image-simplicial-complex}): 
    \begin{align} \label{eq:intersections-proof}
    F_1 \cap F_2 \cap F_3~, \quad 
        &F_1 \cap F_2~, \quad F_1 \cap F_3~, \quad 
    F_2 \cap F_3~,\\
    &F_1 \cap F_4~, \quad F_2 \cap F_4~, \quad F_3 \cap F_4~. \notag
    \end{align}    
    The intersections shown in the second line of~\eqref{eq:intersections-proof} were already seen to be codewords of $\C$, as $\C_{\cap F_4} \subseteq \C$.  
    
    Next, we consider the max-intersection face $\gamma=F_1 \cap F_2 \cap F_3$.  By hypothesis, $\Delta(\C^\ast)$ does not satisfy the Path-of-Facets Condition, so  Lemma~\ref{lem:triple}(a) implies that the link $\link{\Delta(\C^\ast)}{F_1 \cap F_2 \cap F_3}$ is not contractible. Therefore, by definition, $F_1 \cap F_2 \cap F_3$ is a mandatory face of  $\Delta(\C^\ast)$ and so 
    $F_1 \cap F_2 \cap F_3 \in \mincode^\ast \subseteq \C$. 

    Finally, we consider the case when $\gamma$ is nonempty and is one of the following intersections: $F_1 \cap F_2$, $ F_1 \cap F_3$, or
    $F_2 \cap F_3$.  If 
    this pairwise intersection equals the triplewise intersection 
    $F_1 \cap F_2 \cap F_3$, then (as shown above)  $\gamma=F_1 \cap F_2 \cap F_3 \in \C$.  
    So, assume $ \gamma = F_i \cap F_j \not\subseteq F_k$, where $(i,j,k)$ is a permutation of $(1,2,3)$.  We also know that $\gamma = F_i \cap F_j \not\subseteq F_4$, because $F_i \cap F_j \cap F_4 = \varnothing$.   
    Thus, by Lemma \ref{lem:intersection-2-facets}, $\link{\Delta(\C)}{\gamma}$ is not contractible.  We conclude that $\gamma$ is a mandatory face of $\C$, and therefore (since $\C$ has no local obstructions) $\gamma \in \C$. 
    
    We conclude that $\C$ is max-intersection-complete and thus (by Lemma~\ref{Lem:max-intersection})  convex.
\end{proof}

\subsection{Convexity of L18 and L21 codes}
    The main result of this subsection is the following.
    
    \begin{thm} \label{thm: l18-l21}
    Let $\C$ be a $4$-maximal code, and let
    $\mathcal{F}(\C)=\{F_1,F_2,F_3,F_4\}$ denote the set of maximal codewords.
    If $\mathcal{N}(\mathcal{F}(\C))$ is L18 or L21, then $\C$ is convex if and only if it has no local obstructions.
\end{thm}

One direction of Theorem~\ref{thm: l18-l21} holds in general (Lemma~\ref{lem: convex-no-obs} states that convex codes have no local obstructions).  As for the converse, one case  was proven already; specifically, Proposition~\ref{prop: pathoffaces-convex} handles the codes in which the simplicial complex $\Delta(\{F_1,F_2,F_3\})$ does not satisfy the Path-of-Facets Condition.  The remainder of this subsection therefore concerns the cases when the Path-of-Facets Condition holds; we prove this separately for the nerves L18 (Lemma~\ref{lem:L18} below) and L21 (Lemma~\ref{L21}).  
    
The idea behind these proofs is to construct a convex realization of the minimal code  of $\Delta(\C)$ by ``gluing'' together the convex realizations of the subcodes $\mincode^\ast$ and $\C_{\cap F_4}$ defined in Lemma~\ref{lem:restrict-facet}.  
We illustrate such a ``gluing'' in the following example involving L18 codes.

\begin{ex}[L18 codes] \label{ex:L18}
 Consider the following $4$-maximal codes: 
    \begin{equation}
    \label{eq:L18-codes-example}        
    \C_{18}^a = 
    \{\mathbf{345}, \mathbf{234}, \mathbf{356}, \mathbf{12}, 34, 35, 2, \varnothing \} \hspace{15pt} \text{ and } \hspace{15pt} 
    \C_{18}^b = 
    \{\mathbf{123}, \mathbf{1346}, \mathbf{145}, \mathbf{67}, 13, 14, 6, \varnothing \}~. 
    \end{equation}
    It is straightforward to check that both $\C_{18}^a$ and $\C_{18}^b$ are minimal L18 codes, with maximal codewords $F_1,F_2,F_3,F_4$ as listed in the orders shown in~\eqref{eq:L18-codes-example}.  Next, following the notation of Lemma~\ref{lem:restrict-facet}, we have:
    \begin{align}
    \notag
    (\C_{18}^a)_{\cap F_4} ~&=~
     \{\mathbf{12}, 2, \varnothing \}   ~, 
    &
    \quad 
        (\C_{18}^a)^*_{\rm min} ~&=~
    \{\mathbf{345}, \mathbf{234}, \mathbf{356}, 34, 35, \varnothing \}
    ~,
    \\
    \notag
    (\C_{18}^b)_{\cap F_4} ~&=~
    \{\mathbf{67},  6, \varnothing \} ~,
    &
    \quad 
        (\C_{18}^b)^*_{\rm min} ~&=~
    \{\mathbf{123}, \mathbf{1346}, \mathbf{145},  13, 14, \varnothing \}~.
    \end{align}
Convex realizations of $(\C_{18}^a)_{\cap F_4}$ and $ (\C_{18}^a)^*_{\rm min}$ are shown below (the latter realization coming from Figure~\ref{fig:realization-1-dim}):

        \begin{center}
        \begin{tikzpicture}
        \filldraw[black] (2,0) circle (2pt);
        \filldraw[black] (0,0) circle (2pt);
        \filldraw[black] (-2,0) circle (2pt);
        \filldraw[black] (-4,0) circle (2pt);
        \filldraw[black] (-6,0) circle (2pt);
        \filldraw[black] (4,0) circle (2pt);

        \draw[thick, black] (-6,0)--(-4,0) node [midway, fill=white] {$\mathbf{234}$};
        \draw[thick, black] (-4,0)--(-2,0) node [midway, fill=white] {$34$};
        \draw[thick, black] (-2,0)--(0,0) node [midway, fill=white] {$\mathbf{345}$};
        \draw[thick, black] (0,0)--(2,0) node [midway, fill=white] {$35$};
        \draw[thick, black] (2,0)--(4,0) node [midway, fill=white] {$\mathbf{356}$};

        \filldraw[black] (-12,0) circle (2pt);
        \filldraw[black] (-10,0) circle (2pt);
        \filldraw[black] (-8,0) circle (2pt);

        \draw[thick, black] (-12,0)--(-10,0) node [midway, fill=white] {$\mathbf{12}$};    
        \draw[thick, black] (-10,0)--(-8,0) node [midway, fill=white] {$2$};
    \end{tikzpicture}    
    \end{center}
    The above two realizations can be ``glued'' to obtain the following convex realization of $\C_{18}^a$:
        \begin{center}
        \begin{tikzpicture}
        \filldraw[black] (2,0) circle (2pt);
        \filldraw[black] (0,0) circle (2pt);
        \filldraw[black] (-2,0) circle (2pt);
        \filldraw[black] (-4,0) circle (2pt);
        \filldraw[black] (-6,0) circle (2pt);
        \filldraw[black] (4,0) circle (2pt);

        \draw[thick, black] (-6,0)--(-4,0) node [midway, fill=white] {$\mathbf{234}$};
        \draw[thick, black] (-4,0)--(-2,0) node [midway, fill=white] {$34$};
        \draw[thick, black] (-2,0)--(0,0) node [midway, fill=white] {$\mathbf{345}$};
        \draw[thick, black] (0,0)--(2,0) node [midway, fill=white] {$35$};
        \draw[thick, black] (2,0)--(4,0) node [midway, fill=white] {$\mathbf{356}$};

        \filldraw[black] (-10,0) circle (2pt);
        \filldraw[black] (-8,0) circle (2pt);

        \draw[thick, black] (-10,0)--(-8,0) node [midway, fill=white] {$\mathbf{12}$};    
        \draw[thick, black] (-8,0)--(-6,0) node [midway, fill=white] {$2$};
        \end{tikzpicture}    
            \end{center}
Next, convex realizations of $(\C_{18}^b)_{\cap F_4}$ and $ (\C_{18}^b)^*_{\rm min}$ are depicted below: 

    
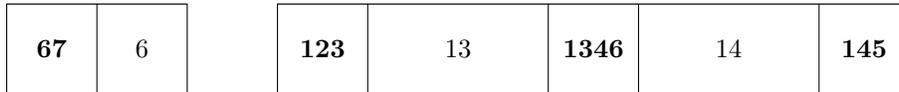
\begin{center}
    \begin{tikzpicture}[scale = 1.2]
    \draw (0,0) rectangle (7,1);

    \draw (1,0) -- (1,1);
    \draw (3,0) -- (3,1);
    \draw (4,0) -- (4,1);
    \draw (6,0) -- (6,1);

    \node at (.5,.5) {$\mathbf{123}$};
    \node at (2,.5) {$13$};
    \node at (3.5,.5) {$\mathbf{1346}$};
    \node at (5,.5) {$14$};
    \node at (6.5,.5) {$\mathbf{145}$};
    %
    \draw (-1,0) rectangle (-3,1);

    \draw (-2,0) -- (-2,1);

    \node at (-2.5, .5) {$\mathbf{67}$};
    \node at (-1.5, .5) {$6$};

    \end{tikzpicture}
    \end{center}

We ``glue'' the above two realizations to obtain a T-shaped convex realization of $\C_{18}^b$, as follows:

    
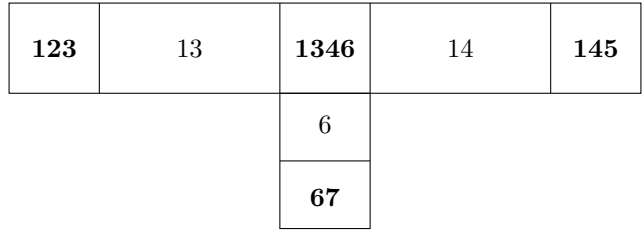
\begin{center}
    \begin{tikzpicture}[scale = 1.2]
    \draw (0,0) rectangle (7,1);

    \draw (1,0) -- (1,1);
    \draw (3,0) -- (3,1);
    \draw (4,0) -- (4,1);
    \draw (6,0) -- (6,1);

    \node at (.5,.5) {$\mathbf{123}$};
    \node at (2,.5) {$13$};
    \node at (3.5,.5) {$\mathbf{1346}$};
    \node at (5,.5) {$14$};
    \node at (6.5,.5) {$\mathbf{145}$};
    %
    \draw (3,0) rectangle (4,-1.5);

    \draw (3,-.75) -- (4,-.75);

    \node at (3.5, -1.15) {$\mathbf{67}$};
    \node at (3.5, -0.35) {$6$};

    \end{tikzpicture}
    \end{center}
\end{ex}

The realizations shown in Example~\ref{ex:L18} for $\C_{18}^a$ and $\C_{18}^b$ represent two of the cases in the proof of Lemma~\ref{lem:L18} below. In particular, the two realizations generalize to the realizations shown later in Figures~\ref{L18CL} and~\ref{fig:L18CT-case-2}, respectively.

    \begin{lemma}[L18] \label{lem:L18}
    Let $\C$ be a $4$-maximal neural code, and let
    $\mathcal{F}(\C)=\{F_1,F_2,F_3,F_4\}$ denote the set of maximal codewords.
    If the following hold:
    \begin{enumerate}
        \item $\C$ has no local obstructions,
        \item the nerve $\mathcal{N}(\mathcal{F}(\C))$ is L18 (in particular, $F_1 \cap F_2 \cap F_3$ is the unique nonempty triplewise-intersection of maximal codewords), and
        \item $\Delta(\{F_1,F_2,F_3\})$ satisfies the Path-of-Facets Condition;
    \end{enumerate}
    then $\C$ is convex. 
    \end{lemma}

    \begin{proof}
Following the notation in Lemma~\ref{lem:restrict-facet}, we define the following, 
where $\C^\ast := \{F_1, F_2, F_3\}$: 
    \[
    \mincode^\ast ~:=~ \mincode(\Delta(\C^\ast))
    \quad \mathrm{and}
    \quad 
    \C_{\cap F_4} 
    ~:=~
    \{F_4\} \cup 
    \{F_1 \cap F_4\} \cup \{ F_2 \cap F_4 \} \cup \{F_3 \cap F_4\}
        ~=~ \{\varnothing, F_2 \cap F_4, F_4\}~,
    \]
    where we used the fact that 
    the nerve $\mathcal{N}(\mathcal{F}(\C))$ is L18 to conclude that $F_1 \cap F_4=F_3 \cap F_4 = \varnothing$ and $F_2 \cap F_4 \neq \varnothing$.    
    (Also, $F_2 \cap F_4 \neq F_4$ follows from the fact that $F_2$ and $F_4$ are both maximal codewords.)

        We claim that 
    $\mincode^\ast \cup \C_{\cap F_4} \subseteq \C$. 
    Indeed, this follows from Lemma~\ref{lem:restrict-facet}, which applies here because $\C$ has no local obstructions and is an L18 code.
    
    Next, 
    by construction, $\mincode^\ast \cup \C_{\cap F_4}$ has the same maximal codewords as $\C$, so $ \C \subseteq \Delta(\mincode^\ast \cup \C_{\cap F_4})$.  We conclude that     
    $\mincode^\ast\cup \C_{\cap F_4} \subseteq \C \subseteq \Delta(\mincode^\ast \cup \C_{\cap F_4})$ and so, by Lemma \ref{lem:monotonicity}, it suffices to prove that the code $\mincode^\ast\cup \C_{\cap F_4} $ is convex.  The remainder of this proof is dedicated to doing so.

It is straightforward to check that 
a convex realization of $\C_{\cap F_4}$ in $\mathbb{R}$ is depicted below:

        \begin{center}
        \begin{tikzpicture}
        \filldraw[black] (-12,0) circle (2pt);
        \filldraw[black] (-10,0) circle (2pt);
        \filldraw[black] (-8,0) circle (2pt);

        \draw[thick, black] (-12,0)--(-10,0) node [midway, fill=white] {$\mathbf{F_4}$};    
        \draw[thick, black] (-10,0)--(-8,0) node [midway, fill=white] {$F_2\cap F_4$};
    \end{tikzpicture}    
    \end{center}

Next, 
    since $\Delta(\C)$ satisfies the Path-of-Facets Condition, 
Lemma~\ref{lem:path}(b) gives 
a convex realization of $\mincode^\ast$ in $\mathbb{R}$ (in Figure \ref{fig:realization-1-dim}).  For convenience, we depict this realization again below, recalling that $(a,b,c)$ is a permutation of $(1,2,3)$  
     such that $(F_a \cap F_b) \smallsetminus F_c \neq \varnothing$,  $(F_b \cap F_c)\smallsetminus F_a \neq \varnothing$, and $(F_a \cap F_c) \smallsetminus F_b = \varnothing$:

\begin{center}
        \begin{tikzpicture}
        \filldraw[black] (2,0) circle (2pt);
        \filldraw[black] (0,0) circle (2pt);
        \filldraw[black] (-2,0) circle (2pt);
        \filldraw[black] (-4,0) circle (2pt);
        \filldraw[black] (-6,0) circle (2pt);
        \filldraw[black] (4,0) circle (2pt);

        \draw[thick, black] (-6,0)--(-4,0) node [midway, fill=white] {$\mathbf{F_a}$};
        \draw[thick, black] (-4,0)--(-2,0) node [midway, fill=white] {$F_a \cap F_b$};
        \draw[thick, black] (-2,0)--(0,0) node [midway, fill=white] {$\mathbf{F_b}$};
        \draw[thick, black] (0,0)--(2,0) node [midway, fill=white] {$F_b \cap F_c$};
        \draw[thick, black] (2,0)--(4,0) node [midway, fill=white] {$\mathbf{F_c}$};
    \end{tikzpicture}       
\end{center}    

We consider cases based on whether $a=2$, $b=2$, or $c=2$ (examples of the first two cases appeared in Example~\ref{ex:L18}).

{\bf Case 1: $a=2$.}
In this case, we claim that the realizations of $\mincode^\ast$ and $\C_{\cap F_4}$ can be glued as shown in Figure~\ref{L18CL} to obtain a convex realization of $\mincode^\ast \cup \C_{\cap F_{4}}$.  
We see this as follows.  Let $n$ denote the number of neurons of $\C$.  For $i\in [n] \smallsetminus F_4$, the open set (open interval) $U_i$ is unchanged from the realization shown earlier for $\mincode^\ast$ (indeed, the fact that $i \notin F_4$ implies that $i$ is not in the two ``added'' codewords $F_4$ and $F_2 \cap F_4$).
Now consider the remaining case: $i \in F_4$.
Notice that $i$ is not in the $4$ right-most codewords ($F_2 \cap F_b$, $F_b$, $F_b \cap F_c$, $F_c$) because $F_4$ intersects only $F_2$ (the L18 assumption is used here).  We consider two subcases based on whether $i\in F_2$.  If $i \in F_2$ (so, $i \in F_2 \cap F_4$), then $U_i$ is the open interval that encompasses the regions labeled by $F_4$, $F_2 \cap F_4$, and $F_2$ in Figure~\ref{L18CL}.  On the other hand, if $i \in F_4 \smallsetminus F_2$, then $U_i$ is the open interval of the region $F_4$ in Figure~\ref{L18CL}.

\begin{figure}[ht]
    \centering
        \begin{tikzpicture}
        \filldraw[black] (2,0) circle (2pt);
        \filldraw[black] (0,0) circle (2pt);
        \filldraw[black] (-2,0) circle (2pt);
        \filldraw[black] (-4,0) circle (2pt);
        \filldraw[black] (-6,0) circle (2pt);
        \filldraw[black] (4,0) circle (2pt);

        \draw[thick, black] (-6,0)--(-4,0) node [midway, fill=white] {$\mathbf{F_2}$};
        \draw[thick, black] (-4,0)--(-2,0) node [midway, fill=white] {$F_2 \cap F_b$};
        \draw[thick, black] (-2,0)--(0,0) node [midway, fill=white] {$\mathbf{F_b}$};
        \draw[thick, black] (0,0)--(2,0) node [midway, fill=white] {$F_b \cap F_c$};
        \draw[thick, black] (2,0)--(4,0) node [midway, fill=white] {$\mathbf{F_c}$};

        \filldraw[black] (-10,0) circle (2pt);
        \filldraw[black] (-8,0) circle (2pt);

        \draw[thick, black] (-10,0)--(-8,0) node [midway, fill=white] {$\mathbf{F_4}$};    
        \draw[thick, black] (-8,0)--(-6,0) node [midway, fill=white] {$F_2 \cap F_4$};
        \end{tikzpicture}    
    \caption{Convex realization of $\mincode^\ast \cup \C_{\cap F_{4}}$ (Case 1).}
    \label{L18CL}
\end{figure}

{\bf Case 2: $b=2$.}
In this case, we claim that 
the convex realizations of $\mincode^\ast$ and $\C_{\cap F_4}$ shown earlier can be ``extended'' into $\mathbb{R}^2$ and then glued together as in Figure \ref{fig:L18CT-case-2} to obtain a convex realization. 
The proof for this case proceeds similarly to that for Case 1.  
For $i\in [n] \smallsetminus F_4$, the convex open set $U_i$ is unchanged from that in the ``extended'' realization for $\mincode^\ast$.  
Now assume $i \in F_4$.
The fact that $F_4$ intersects only $F_2$ implies that $i$ is not in the ``left side'' codewords $F_a$ and $F_a \cap F_2$, nor in the ``right side'' codewords $F_2 \cap F_c$ and $F_c$.  If $i \in F_2$ (so, $i \in F_2 \cap F_4$), then $U_i$ is the open rectangle that encompasses the regions labeled by $F_2$, $F_2 \cap F_4$, and $F_4$ in Figure~\ref{fig:L18CT-case-2}.  On the other hand, if $i \in F_4 \smallsetminus F_2$, then $U_i$ is the open rectangle labeled by $F_4$ in Figure~\ref{fig:L18CT-case-2}.

\begin{figure}[ht]
\centering
\begin{tikzpicture}[scale = 1.2]
    \draw (0,0) rectangle (7,1);

    \draw (1,0) -- (1,1);
    \draw (3,0) -- (3,1);
    \draw (4,0) -- (4,1);
    \draw (6,0) -- (6,1);

    \node at (.5,.5) {$\mathbf{F_a}$};
    \node at (2,.5) {$F_a \cap F_2$};
    \node at (3.5,.5) {$\mathbf{F_2}$};
    \node at (5,.5) {$F_2 \cap F_c$};
    \node at (6.5,.5) {$\mathbf{F_c}$};

    \draw (3,-2) -- (3,0);
    \draw (4,-2) -- (4,0);
    \draw (3,-2) -- (4,-2);

    \draw (3,-1) -- (4,-1);

    \node at (3.5,-0.5) {$F_2 \cap F_4$};
    \node at (3.5,-1.5) {$\mathbf{F_4}$};
    
\end{tikzpicture}
\caption{T-shaped convex realization of $\C_{\min} \cup \C_{\cap F_{4}}$ (Case 2).}
    \label{fig:L18CT-case-2}
\end{figure}
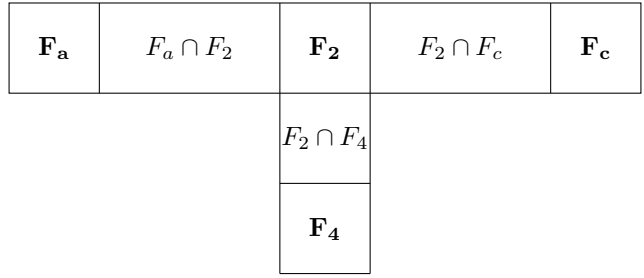

{\bf Case 3: $c=2$.}  This case is symmetric to the $a=2$ case (Case 1).
\end{proof}

\begin{lemma}[L21] \label{L21}
    Let $\C$ be a $4$-maximal neural code, and let
    $\mathcal{F}(\C)=\{F_1,F_2,F_3,F_4\}$ denote the set of maximal codewords.
    If the following hold:
    \begin{enumerate}
        \item $\C$ has no local obstructions,
        \item the nerve $\mathcal{N}(\mathcal{F}(\C))$ is L21 (in particular, $F_1 \cap F_2 \cap F_3$ is the unique nonempty triplewise-intersection of maximal codewords), and
        \item $\Delta(\{F_1,F_2,F_3\})$ satisfies the Path-of-Facets Condition;
    \end{enumerate}
    then $\C$ is convex. 
\end{lemma}

\begin{proof}
As we did in the proof of Lemma~\ref{lem:L18}, we begin by defining the following (with notation as in Lemma~\ref{lem:restrict-facet}),  
where $\C^\ast := \{F_1, F_2, F_3\}$: 
    \[
    \mincode^\ast ~:=~ \mincode(\Delta(\C^\ast))
    \quad \mathrm{and}
    \quad 
    \C_{\cap F_4} 
    ~:=~
    \{F_4\} \cup 
    \{F_1 \cap F_4\} \cup \{ F_2 \cap F_4 \} \cup \{F_3 \cap F_4\}
        ~=~ \{\varnothing,~ F_2 \cap F_4,~ F_3 \cap F_4, ~F_4\}~.
    \]
    Here we used the fact that the nerve $\mathcal{N}(\mathcal{F}(\C))$ is L21 to conclude that 
$F_1 \cap F_4=\varnothing$, while 
$F_2 \cap F_4 $
 and $F_3 \cap F_4 $ are nonempty and unequal (that is, $F_2 \cap F_4  \neq F_3 \cap F_4 $).  Also, neither $F_2 \cap F_4 $
 nor $F_3 \cap F_4 $ is equal to $F_4$, as the sets $F_i$ are maximal codewords.

For essentially the same reasons as in the proof of Lemma~\ref{lem:L18} (in particular, Lemma~\ref{lem:monotonicity} applies both to L18 and L21 codes), it suffices to prove that the code $\mincode^\ast\cup \C_{\cap F_4} $ is convex.  To do so, we will ``glue'' together convex realizations of 
    $\mincode^\ast$ and $\C_{\cap F_4} $.

We begin with $\C_{\cap F_4}  = \{\varnothing,~ F_2 \cap F_4,~ F_3 \cap F_4, ~F_4\}$.  It is straightforward to check that the two realizations of $\C_{\cap F_4}$ shown in Figure~\ref{fig:L21PCF4} are both convex realizations.  
    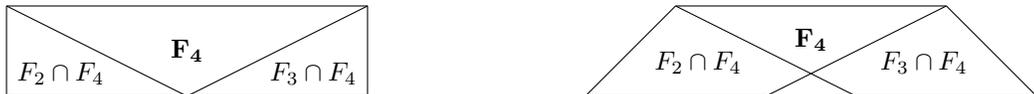
\begin{figure}[ht]
        \centering
        \begin{subfigure}{.5\textwidth}
        \centering
        \begin{tikzpicture}[scale = 1.2]
        \draw (0,1) -- (0,2);
        \draw (0,2) -- (4,2);
        \draw (4,1) -- (4,2);
    
        \draw (2,1) -- (4,2);
        \draw (2,1) -- (0,2);
        \draw (0,1) -- (4,1);
        
        \node at (2,1.5) {$\mathbf{F_{4}}$};
        \node at (0.6,1.25) {$F_{2} \cap F_{4}$};
        \node at (3.4,1.25) {$F_{3} \cap F_{4}$};
    \end{tikzpicture}
        \end{subfigure}%
        \begin{subfigure}{.5\textwidth}
        \centering
                \begin{tikzpicture}[scale = 0.3]
        \draw (4,4) -- (12,4);
        \draw (16,4) -- (24,4);
        \draw (8,8) -- (20,8);
    
        \draw (4,4) -- (8,8);
        \draw (12,4) -- (20,8);
        \draw (8,8) -- (16,4);
        \draw (20,8) -- (24,4);
    
        \node at (9,5.5) {$F_{2} \cap F_{4}$};e
        \node at (14,6.5) {$\mathbf{F_{4}}$};
        \node at (19,5.5) {$F_{3} \cap F_{4}$};
    \end{tikzpicture}
        \end{subfigure}
        \caption{Two convex realizations of $\C_{\cap F_4}$.}
        \label{fig:L21PCF4}

    \end{figure}

    Next, we consider $\mincode^\ast$.  
    By hypothesis, $\Delta(\C^\ast)$ satisfies the Path-Of-Facets Condition. So, by Lemma~\ref{lem:path}(b), $\mincode^\ast$ is convex and can be realized in $\mathbb{R}$ as shown earlier in Figure~\ref{fig:realization-1-dim}. Recall that the labeling of this realization depends on a
    permutation     
    $(a,b,c)$ of $(1,2,3)$  
     such that $(F_a \cap F_b) \smallsetminus F_c$ and $(F_b \cap F_c)\smallsetminus F_a$ are nonempty, while $(F_a \cap F_c) \smallsetminus F_b$ is empty.  Using this notation, consider the $b=2$ and $b=1$ cases; corresponding convex realizations (extended into $\mathbb{R}^2$) of $\mincode^\ast$ are shown in Figure~\ref{fig:L21-extend-into-2-d}.
     
    \begin{figure}[ht]
        \centering
        \begin{subfigure}{.5\textwidth}
            \centering
                \begin{tikzpicture}[scale = 1.2]
        \draw (-2,0) -- (0,1);
        \draw (0,0) -- (2,1);
    
        \draw (2,1) -- (4,0);

        \draw (-2,1) -- (4,1);
        \draw (-2,0) -- (4,0);

        \draw (-2,0) -- (-2,1);
        \draw (0,0) -- (0,1);
        \draw (4,0) -- (4,1);
        
        \node at (-1.4,0.7) {$\mathbf{F_{1}}$};
        \node at (-0.6,0.3) {$F_{1} \cap F_{2}$};
        \node at (2,0.3) {$F_{2} \cap F_{3}$};
        \node at (0.6,0.75) {$\mathbf{F_{2}}$};
        \node at (3.4,0.75) {$\mathbf{F_{3}}$};
    \end{tikzpicture}
        \end{subfigure}%
        \begin{subfigure}{.5\textwidth}
        \centering
            \begin{tikzpicture}[scale = 0.3]
        \draw (0,0) -- (28,0);
        \draw (4,4) -- (24,4);
    
        \draw (12,0) -- (12,4);
        \draw (16,0) -- (16,4);
    
        \draw (0,0) -- (4,4);
        \draw (4,0) -- (12,4);
        \draw (16,4) -- (24,0);
        \draw (24,4) -- (28,0);
    
        \node at (4,2) {$\mathbf{F_{2}}$};e
        \node at (9.75,1.25) {$F_{1} \cap F_{2}$};e
        \node at (14,2) {$\mathbf{F_{1}}$};
        \node at (18.25,1.25) {$F_{1} \cap F_{3}$};
        \node at (24,2) {$\mathbf{F_{3}}$};
    \end{tikzpicture}
        \end{subfigure}
        \caption{Extensions of the realization of $\mincode^\ast$ in Figure~\ref{fig:realization-1-dim} into $\R^2$ in the cases of $b=2$ (left) and $b=1$ (right).}
        \label{fig:L21-extend-into-2-d}
    \end{figure}
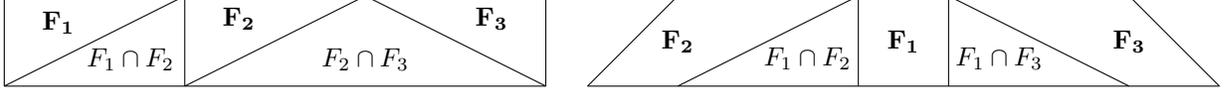

    Next, we glue together realizations of $\mincode^\ast$  and $\C_{\cap F_4}$ (from Figures~\ref{fig:L21PCF4} and~\ref{fig:L21-extend-into-2-d}, respectively).  
    Specifically, for the cases $b=2$ and $b=1$,
    we obtain the 
    realizations shown on the left and right (respectively) in Figure~\ref{fig:L21Pconvexrealization}.
    Finally, the $b=3$ case is symmetric to the $b=1$ case.

     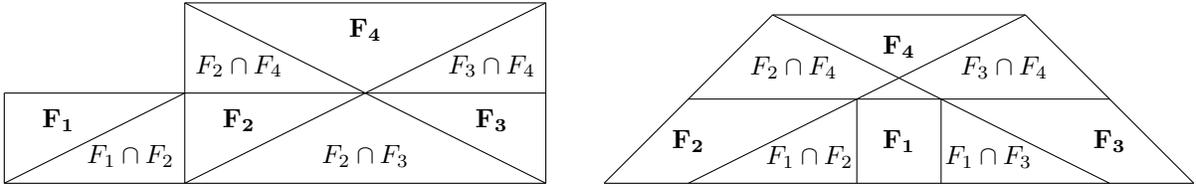
\begin{figure}[ht]
        \centering
        \begin{subfigure}{.5\textwidth}
            \centering
                \begin{tikzpicture}[scale = 1.2]
        \draw (0,0) -- (0,2);
        \draw (0,2) -- (4,2);
        \draw (4,2) -- (4,0);
    
        \draw (0,0) -- (4,2);
        \draw (0,2) -- (4,0);
        \draw (0,1) -- (4,1);
        
        \node at (2,1.7) {$\mathbf{F_{4}}$};
        \node at (0.6,1.3) {$F_{2} \cap F_{4}$};
        \node at (3.4,1.3) {$F_{3} \cap F_{4}$};
        \node at (0.6,0.7) {$\mathbf{F_{2}}$};
        \node at (3.4,0.7) {$\mathbf{F_{3}}$};
    
        \draw (-2,1) -- (0,1);
        \draw (-2,0) -- (4,0);
        
        \draw (-2,0) -- (-2,1);
    
        \draw (-2,0) -- (0, 1);
    
        \node at (-1.4,0.7) {$\mathbf{F_{1}}$};
        \node at (-0.6,0.3) {$F_{1} \cap F_{2}$};
        \node at (2,0.3) {$F_{2} \cap F_{3}$};
    \end{tikzpicture}
        \end{subfigure}%
        \begin{subfigure}{.5\textwidth}
        \centering
            \begin{tikzpicture}[scale = 0.28]
        \draw (0,0) -- (28,0);
        \draw (4,4) -- (24,4);
        \draw (8,8) -- (20,8);
    
        \draw (12,0) -- (12,4);
        \draw (16,0) -- (16,4);
    
        \draw (0,0) -- (8,8);
        \draw (4,0) -- (20,8);
        \draw (8,8) -- (24,0);
        \draw (20,8) -- (28,0);
    
        \node at (9,5.5) {$F_{2} \cap F_{4}$};e
        \node at (14,6.5) {$\mathbf{F_{4}}$};
        \node at (19,5.5) {$F_{3} \cap F_{4}$};
        \node at (4,2) {$\mathbf{F_{2}}$};e
        \node at (9.75,1.25) {$F_{1} \cap F_{2}$};e
        \node at (14,2) {$\mathbf{F_{1}}$};
        \node at (18.25,1.25) {$F_{1} \cap F_{3}$};
        \node at (24,2) {$\mathbf{F_{3}}$};
    \end{tikzpicture}
        \end{subfigure}
        \caption{
        Result of gluing the realizations of $\C_{\cap F_4}$ and $\C_{\min}^\ast$ in Figures~\ref{fig:L21PCF4} and~\ref{fig:L21-extend-into-2-d} in the cases of $b=2$ and $b=1$, respectively.
        }
        \label{fig:L21Pconvexrealization}
        \label{fig:L21Dconvexrealization}
    \end{figure}
    
To complete the proof, we need only show that the realizations shown in Figure \ref{fig:L21Pconvexrealization} are indeed convex realizations.  The explanation is essentially the same for the two realizations, 
as follows.
Let $n$ denote the number of neurons of $\C$.  For $i\in [n] \smallsetminus F_4$, the open set $U_i$ is unchanged from the realization shown earlier for $\mincode^\ast$ (indeed, the fact that $i \notin F_4$ implies that $i$ is not in any of the three ``added'' codewords $F_4$, $F_2 \cap F_4$, and $F_3 \cap F_4$).
Now consider the remaining case: $i \in F_4$.
Notice that $i$ is not in any of the 
nonempty codewords of $\mincode^\ast$ (those codewords shown in the ``bottom level'' of the realizations); the L21 assumption is used here.  We consider three subcases based on whether $i \in F_2$, $i\in F_3$, or $i \notin (F_2 \cup F_3)$.  (It is not possible for $i \in F_2 \cap F_3$, because $F_2 \cap F_3 \cap F_4 = \varnothing$ due to the L21 assumption.)  If $i\in F_2$ (respectively, $i\in F_3$), then $U_i$ is the open set that encompasses the regions labeled by $F_4$,  $F_2 \cap F_4$, and $F_2$ (respectively, $F_4$,  $F_3 \cap F_4$, and $F_3$) and no other regions (due to the L21 assumption), and this set $U_i$ is readily seen to be convex.   On the other hand, if $i \in F_4 \smallsetminus (F_2 \cup F_3)$, then $U_i$ is the open triangular region labeled by $F_4$ and thus is convex.    
    \end{proof}

\begin{remark}[Minimal codes for L18 and L21]
It follows from the proofs of Lemmas~\ref{lem:L18} and~\ref{L21} that, for L18 and L21 codes $\C$, the minimal code of $\Delta(\C)$ is precisely equal to $\mincode^\ast \cup \C_{\cap F_4}$. 
\end{remark}

\subsection{Convexity of L22 codes} \label{sec:L22}
The first simplicial complex in Figure~\ref{fig:image-simplicial-complex} with more than one $2$-simplex is L22. Despite this added complexity, the convexity of L22 codes is again characterized by the presence of local obstructions (Theorem~\ref{thm: l22} below).  We first need two lemmas, as follows.

\begin{lemma}[Pairwise intersections] \label{lem:L22pairwise}
    Let $\C$ be a $4$-maximal code, and let $\mathcal{F}(\C)=\{F_1,F_2,F_3,F_4\}$ denote the set of maximal codewords. 
    Assume that $\mathcal{N}(\mathcal{F}(\C))$ is L22, and consider distinct $i,j \in \{1,2,3,4\}$. 
    If $F_i \cap F_j \neq F_1 \cap F_2 \cap F_3$ and $F_i \cap F_j \neq F_2 \cap F_3 \cap F_4$, then $F_i \cap F_j \in \mincode$, 
    where $\mincode$ is the minimal code of~$\Delta(\C)$.
\end{lemma}
\begin{proof}
    We prove the contrapositive.  Assume that $F_i \cap F_j \notin \mincode$.  By definition, this means that $F_i \cap F_j \neq \varnothing$ and the link of $F_i \cap F_j $ in $\Delta(\C)= \Delta(\mincode)$ is contractible.  Now Lemma~\ref{lem:intersection-2-facets} implies that $(F_i \cap F_j) \subseteq F_k$ for some  $k \in \{1,2,3,4\} \smallsetminus \{i,j\}$ and so 
    $F_i \cap F_j= F_i \cap F_j \cap F_k$.  Hence, $\varnothing \neq F_i \cap F_j= F_i \cap F_j \cap F_k$.  However, $F_1\cap F_2 \cap F_3 $ and $F_2\cap F_3 \cap F_4 $ are the only \uline{nonempty} triplewise intersections of the maximal codewords $F_i$ (because $\mathcal{N}(\mathcal{F})$ is L22).  So, 
    $F_i \cap F_j= F_1 \cap F_2 \cap F_3$ or $F_i \cap F_j= F_2 \cap F_3 \cap F_4$, which completes the proof.
\end{proof}
\begin{lemma}[L22 minimal codes] \label{lem:L22min}
    Let $\C$ be a $4$-maximal code, and let $\mathcal{F}(\C)=\{F_1,F_2,F_3,F_4\}$ denote the set of maximal codewords. 
        Let $\mincode$ denote the minimal code of $\Delta(\C)$.  
    If $\mathcal{N}(\mathcal{F}(\C))$ is L22, then $\mincode = (\C_1^\ast)_{\min} \cup (\C_4^\ast)_{\min}$, where 
    $\mathcal{C}^*_1 := \{F_1, F_2, F_3\}$
    and $ 
    \mathcal{C}^*_4 := \{F_2, F_3, F_4\}$.          
\end{lemma}

\begin{proof}
     Our proof proceeds by examining each of the possible codewords in $\mincode$. Every nonempty codeword in $\mincode$ is either a facet $F_i$ of $\Delta(\C)$ or the intersection of two or more facets of $\Delta(\C)$ (recall equation~\eqref{eq:c-min} and Lemma~\ref{lem:intersection-facet}). By definition, the codewords $\varnothing, F_1, F_2, F_3, F_4$ are always contained in both $\mincode$ and the union $(\C^\ast_1)_{\min} \cup (\C^\ast_4)_{\min}$. The remainder of the proof therefore examines the intersections of facets of $\Delta(\C)$. 

    Let $\sigma \in \{F_1 \cap F_2, F_1 \cap F_3, F_1 \cap F_2 \cap F_3\}$. Note that $\sigma \cap F_4 = \varnothing$
    (this is because $\mathcal{N}(\mathcal{F}(\C))$ is L22, which implies that the only nonempty triplewise intersection involving $F_4$ is $F_2 \cap F_3 \cap F_4$). Consider the facets of the link $\link{\Delta(\C_1^\ast)}{\sigma}$, namely,

        \[\mathcal{L}_{\Delta((\C_1^\ast)_{\min})}(\sigma) = \{F \smallsetminus \sigma : F \text{ is a facet of } \Delta((\C_1^\ast)_{\min}) \text{ containing } \sigma \}.\]

        Every facet $F\smallsetminus\sigma \in \mathcal{L}_{\Delta((\C_1^\ast)_{\min})}(\sigma)$ is also a facet in $\mathcal{L}_{\Delta(\mincode)}(\sigma)$, since such an $F$ is also a facet of $\mincode$. Furthermore, $F_4 \smallsetminus \sigma$ cannot be in $\mathcal{L}_{\Delta(\mincode)}(\sigma)$, because $\sigma \cap F_4 = \varnothing$. Therefore, $\mathcal{L}_{\Delta(\mincode)}(\sigma) = \mathcal{L}_{\Delta((\C_1^\ast)_{\min})}(\sigma)$. Since the link itself is a simplicial complex, and both $\link{\Delta(\C_1^\ast)}{\sigma}$ and $\link{\Delta(\mincode)}{\sigma}$ have the same facets, we have $\link{\Delta(\C_1^\ast)}{\sigma} = \link{\Delta(\mincode)}{\sigma}$. We then conclude that $\sigma \in (\C_1^\ast)_{\min}$ if and only if $\sigma \in \mincode$. A similar argument holds for $\sigma \in \{F_2 \cap F_4, F_3 \cap F_4, F_2 \cap F_3 \cap F_4\}$; we conclude that $\sigma \in (\C_4^\ast)_{\min}$ if and only if $\sigma \in \mincode$.

         The final 
         intersection of facets to consider is $F_2 \cap F_3$. Because $\C$ is an L22 code, we must have $F_2 \cap F_3\not\subseteq F_1$ and $F_2 \cap F_3 \not\subseteq F_4$. By Lemma \ref{lem:intersection-2-facets}, this implies that $F_2 \cap F_3$ is always in each of $(\C_1^\ast)_{\min}, (\C_4^\ast)_{\min}$, and $\mincode$. 
         We conclude that
         $(\C^\ast_1)_{\min} \cup (\C^\ast_4)_{\min}= \mincode$.  
\end{proof}

In the proof of the following result, some of the key ideas
were mentioned earlier in Example~\ref{ex:path-facets-L22}.

\begin{thm}[L22] \label{thm: l22}
    Let $\C$ be a $4$-maximal code, and let
    $\mathcal{F}(\C)=\{F_1,F_2,F_3,F_4\}$ denote the set of maximal codewords.
    If $\mathcal{N}(\mathcal{F}(\C))$ is L22, then $\C$ is convex if and only if $\C$ has no local obstructions.
\end{thm}
\begin{proof}
    Convex codes have no local obstructions (Lemma~\ref{lem: convex-no-obs}), so it suffices to prove the converse. Accordingly, assume that $\C$ is an L22 code with no local obstructions.  By Lemma~\ref{lem:monotonicity}, we need only show that the minimal code of $\Delta(\C)$, which we denote by $\mincode$, is convex.

    From Figure~\ref{fig:image-simplicial-complex},
    we see that L22 has exactly two facets, which are the following $2$-simplices:
        \[
    \mathcal{C}^*_1 ~:=~ \{F_1, F_2, F_3\} \quad \mathrm{and} \quad \mathcal{C}^*_4 ~:=~ \{F_2, F_3, F_4\}~.          
        \]
Therefore, since the nerve $\mathcal{N}(\mathcal{F}(\C))$ is L22, we have that $F_1\cap F_2 \cap F_3 \neq \varnothing$ and $F_2\cap F_3 \cap F_4 \neq \varnothing$, while all other triplewise intersections of the facets $F_i$ are empty.

    Our proof proceeds by investigating which codewords appear in $\mincode$, beginning with the triplewise intersections.
    We have four cases (below) based on whether each of $F_1\cap F_2 \cap F_3$ and $F_2\cap F_3 \cap F_4$ is a codeword. 
   By Lemmas \ref{lem:path} and \ref{lem:L22min}, $F_1 \cap F_2 \cap F_3$ is a codeword of $\mincode$ if and only if $\Delta((\C_1^\ast)_{\min})$ does not satisfy the Path-of-Facets Condition. 
    Analogously (due to the fact that L22 is symmetric via switching $F_1$ and $F_4$), $F_2 \cap F_3 \cap F_4$ is a codeword of $\mincode$ if and only if  $\Delta((\C_4^\ast)_{\min})$ does not satisfy the Path-of-Facets Condition.  Another fact that we will use below is that $(\C^\ast_1)_{\min} \cup (\C^\ast_4)_{\min} = \C_{\min}$ (which is due to Lemma~\ref{lem:L22min}).

    \textbf{Case 1:} Both $F_1 \cap F_2 \cap F_3$ and $F_2 \cap F_3 \cap F_4$ are codewords of $\mincode$.  In this case, 
    all triplewise intersections of facets are codewords of $\mincode$ (and the quadruplewise intersection $F_1\cap F_2 \cap F_3 \cap F_4$ is empty and thus is a codeword, too).  Additionally, by Lemma  \ref{lem:L22pairwise}, pairwise intersections that are not triplewise intersections are also codewords.  Therefore, $\mincode$ is max-intersection-complete and so (by Lemma~\ref{Lem:max-intersection}) is convex.

    \textbf{Case 2:} Neither $F_1 \cap F_2 \cap F_3$ nor $F_2 \cap F_3 \cap F_4$ is a codeword of $\mincode$.
    As noted above, this means that both simplicial complexes $\Delta((\C^\ast_1)_{\min})$ and $\Delta((\C^\ast_4)_{\min})$ satisfy the Path-Of-Facets Condition
        and therefore have convex realizations shown in Figure~\ref{fig:L22-case-2-beginning}
where 
    $(a_1,b_1,c_1)$ is a permutation of $(1,2,3)$  
     such that $(F_{a_1} \cap F_{b_1}) \smallsetminus F_{c_1} \neq \varnothing$,  $(F_{b_1} \cap F_{c_1})\smallsetminus F_{a_1} \neq \varnothing$, and $(F_{a_1} \cap F_{c_1}) \smallsetminus F_{b_1} = \varnothing$; and 
    analogously $(a_4,b_4,c_4)$ is a permutation of $(2,3,4)$ such that the same requirement holds -- except with the index $1$ replaced by $4$.

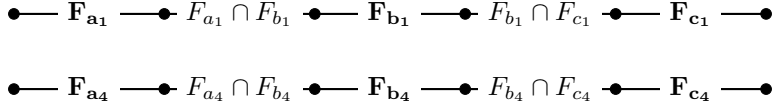
\begin{figure}[ht]
\begin{center}
        \begin{tikzpicture}[scale=1]
        \filldraw[black] (2,0) circle (2pt);
        \filldraw[black] (0,0) circle (2pt);
        \filldraw[black] (-2,0) circle (2pt);
        \filldraw[black] (-4,0) circle (2pt);
        \filldraw[black] (-6,0) circle (2pt);
        \filldraw[black] (4,0) circle (2pt);

        \draw[thick, black] (-6,0)--(-4,0) node [midway, fill=white] {$\mathbf{F_{a_1}}$};
        \draw[thick, black] (-4,0)--(-2,0) node [midway, fill=white] {$F_{a_1} \cap F_{b_1}$};
        \draw[thick, black] (-2,0)--(0,0) node [midway, fill=white] {$\mathbf{F_{b_1}}$};
        \draw[thick, black] (0,0)--(2,0) node [midway, fill=white] {$F_{b_1} \cap F_{c_1}$};
        \draw[thick, black] (2,0)--(4,0) node [midway, fill=white] {$\mathbf{F_{c_1}}$};

        \filldraw[black] (2,-1) circle (2pt);
        \filldraw[black] (0,-1) circle (2pt);
        \filldraw[black] (-2,-1) circle (2pt);
        \filldraw[black] (-4,-1) circle (2pt);
        \filldraw[black] (-6,-1) circle (2pt);
        \filldraw[black] (4,-1) circle (2pt);

        \draw[thick, black] (-6,-1)--(-4,-1) node [midway, fill=white] {$\mathbf{F_{a_4}}$};
        \draw[thick, black] (-4,-1)--(-2,-1) node [midway, fill=white] {$F_{a_4} \cap F_{b_4}$};
        \draw[thick, black] (-2,-1)--(0,-1) node [midway, fill=white] {$\mathbf{F_{b_4}}$};
        \draw[thick, black] (0,-1)--(2,-1) node [midway, fill=white] {$F_{b_4} \cap F_{c_4}$};
        \draw[thick, black] (2,-1)--(4,-1) node [midway, fill=white] {$\mathbf{F_{c_4}}$};
    
    \end{tikzpicture}       
\end{center}    
    \caption{Convex realizations $(\C^\ast_1)_{\min}$ and $(\C^\ast_4)_{\min}$ in Case~2.}
    \label{fig:L22-case-2-beginning}
\end{figure}

Next, we claim that $b_1 \neq 1$ and $b_4\neq 4$ (or, equivalently,
        $b_1 \in \{2,3\}$ and $b_4 \in \{2,3\}$).
We verify this claim as follows.  If $b_1 =1$, then $(F_2 \cap F_3) \smallsetminus F_1 = \varnothing $, which would imply that $(F_2 \cap F_3) \subseteq F_1$ and so the nonempty set $F_2 \cap F_3 \cap F_4$ would be a subset of $F_1$, which contradicts the fact that there are no quadruplewise intersections of facets $F_i$ (since the nerve is L22).  The $b_4=4$ case is symmetric.
        
The simplicial complex L22 is symmetric under switching $F_2$ and $F_3$, so without loss of generality we assume that $b_1=3$.  We thus have two subcases based on whether $b_4=2$ or $b_4=3$, as summarized in Figure~\ref{fig:L22-proof-show-subcases}. 

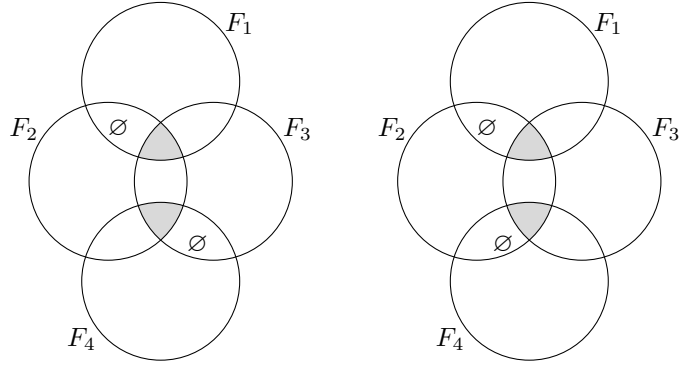
\begin{figure}[ht]
    \centering
\begin{tikzpicture}[scale=0.7]
\draw (1,1.9) circle(1.5cm);
\draw (0,0) circle(1.5cm);
\draw (2,0) circle(1.5cm);
\draw (1,-1.9) circle(1.5cm);
  
\node at (-1.6,1) {$F_2$};
\node at (3.6,1) {$F_3$};
\node at (2.5,3) {$F_1$};
\node at (-0.5,-3) {$F_4$};
\node at (1.7,-1.2) {$\varnothing$};
\node at (0.2,1) {$\varnothing$};
%

\begin{scope}
    \clip (2,0) circle(1.5cm);
    \clip (0,0) circle(1.5cm);
    \fill[opacity=0.15](1,1.9) circle(1.5cm);
\end{scope}

\begin{scope}
    \clip (2,0) circle(1.5cm);
    \clip (0,0) circle(1.5cm);
    \fill[opacity=0.15](1,-1.9) circle(1.5cm);
\end{scope}

\draw (7+1,1.9) circle(1.5cm);
\draw (7+0,0) circle(1.5cm);
\draw (7+2,0) circle(1.5cm);
\draw (7+1,-1.9) circle(1.5cm);
  
\node at (7+-1.6,1) {$F_2$};
\node at (7+3.6,1) {$F_3$};
\node at (7+2.5,3) {$F_1$};
\node at (7+-0.5,-3) {$F_4$};
\node at (7+0.5,-1.2) {$\varnothing$};
\node at (7+0.2,1) {$\varnothing$};
%

\begin{scope}
    \clip (7+2,0) circle(1.5cm);
    \clip (7+0,0) circle(1.5cm);
    \fill[opacity=0.15](7+1,1.9) circle(1.5cm);
\end{scope}

\begin{scope}
    \clip (7+2,0) circle(1.5cm);
    \clip (7+0,0) circle(1.5cm);
    \fill[opacity=0.15](7+1,-1.9) circle(1.5cm);
\end{scope}

\end{tikzpicture}
    \caption{Subcase 2(a) (left) and subcase 2(b) (right).
    The shaded regions indicate the two nonempty triplewise intersections, $F_1 \cap F_2 \cap F_3$ and $F_2 \cap F_3 \cap F_4$.
    }
    \label{fig:L22-proof-show-subcases}
\end{figure}
    \textit{Case 2a:} $b_4=2$.
    In this case, the realizations of 
     $(\C^\ast_1)_{\min}$ and $(\C^\ast_4)_{\min}$ shown earlier in Figure~\ref{fig:L22-case-2-beginning} can be glued to 
     form the realization of $(\C^\ast_1)_{\min} \cup (\C^\ast_4)_{\min} = \C_{\min}$ shown below:

    \begin{center}
        \begin{tikzpicture}
        \filldraw[black] (6,-1) circle (2pt);
        \filldraw[black] (8,-1) circle (2pt);
        \filldraw[black] (4,-1) circle (2pt);
        \filldraw[black] (2,-1) circle (2pt);
        \filldraw[black] (0,-1) circle (2pt);
        \filldraw[black] (-2,-1) circle (2pt);
        \filldraw[black] (-4,-1) circle (2pt);
        \filldraw[black] (-6,-1) circle (2pt);

        \draw[thick, black] (-6,-1)--(-4,-1) node [midway, fill=white] {$\mathbf{F_1}$};
        \draw[thick, black] (-4,-1)--(-2,-1) node [midway, fill=white] {$F_1 \cap F_3$};
        \draw[thick, black] (-2,-1)--(0,-1) node [midway, fill=white] {$\mathbf{F_3}$};
        \draw[thick, black] (0,-1)--(2,-1) node [midway, fill=white] {$F_2 \cap F_3$};
        \draw[thick, black] (2,-1)--(4,-1) node [midway, fill=white] {$\mathbf{F_2}$};
        \draw[thick, black] (4,-1)--(6,-1) node
        [midway, fill=white] {$F_2 \cap F_4$};
        \draw[thick, black] (6,-1)--(8,-1) node [midway, fill=white] {$\mathbf{F_4}$};
    \end{tikzpicture}        
    \end{center}

We verify that this realization is a convex realization as follows.  For $i\in [n] \smallsetminus F_4$ (where $n$ is the number of neurons), the open set (open interval) $U_i$ is unchanged from the realization shown earlier (top of Figure~\ref{fig:L22-case-2-beginning}).

Now consider the remaining case: $i \in F_4$.
From Figure~\ref{fig:L22-proof-show-subcases} (on the left), we see three possibilities: (i) $i\in F_4 \smallsetminus F_2$, (ii) $i \in (F_4 \cap F_2) \smallsetminus F_3$, and (iii) $i \in F_2 \cap F_3 \cap F_4$.  In these situations, the corresponding region $U_i$ is the open interval covering the regions labeled above by, respectively, (i) $F_4$, (ii) $F_2,F_2 \cap F_4, F_4$, (iii) $F_3,F_2 \cap F_3, F_2, F_2 \cap F_4, F_4$.  
Thus, $U_i$ is convex.

    \textit{Case 2b:} 
    $b_4=3$.
    In this case, our aim is first to extend the realizations in Figure~\ref{fig:L22-case-2-beginning} into $\mathbb{R}^2$, but carefully because we cannot glue together the convex realizations as we did in Figure \ref{fig:L18CT-case-2} (for example, there exists neuron $i$ that are both in $F_1 \cap F_3$ and $F_2 \cap F_3$ and neurons $j$ in both $F_3 \cap F_4$, $F_3$, and $F_2 \cap F_3$). 
    We proceed as follows: For $(\C_1^\ast)_{\min}$, we 
    begin with the realization in Figure~\ref{fig:L22-case-2-beginning} and then take the product with an open interval and then ``cut off'' the lower-right corner to obtain the quadrilateral-shaped realization shown in the top part of Figure~\ref{fig:case2(b)-l22} (the regions labeled by $F_1, F_1\cap F_3,F_3,F_2 \cap F_3,F_2$).  For 
    $(\C_4^\ast)_{\min}$, we use a triangle-shaped realization appearing on the right side of Figure~\ref{fig:L22-case-2-beginning} (the regions labeled by $F_4, F_3\cap F_4,F_3, F_2 \cap F_3,F_2$). It is straightforward to check that both extensions are still convex realizations.
    
    Now consider these realizations glued together 
    as shown in Figure \ref{fig:case2(b)-l22}. 
     We must now verify that this realization of $(\C^\ast_1)_{\min} \cup (\C^\ast_4)_{\min} = \C_{\min}$ is a convex realization. For $i\in [n] \smallsetminus F_4$, the open set $U_i$ is the same as the one in the realization of
     $(\C_1^\ast)_{\min}$ we constructed.  
  
     For the remaining case, $i \in F_4$, we see from Figure~\ref{fig:L22-proof-show-subcases} (on the right), that there are three possibilities: 
    (i) $i\in F_4 \smallsetminus F_3$, 
    (ii) $i \in (F_4 \cap F_3) \smallsetminus F_2$, and (iii) $i \in F_2 \cap F_3 \cap F_4$.  In these situations, the corresponding region $U_i$ is the 
    interior of the convex polygon covering the regions labeled above by, respectively, 
    (i) $F_4$, 
    (ii) $F_3,F_3 \cap F_4, F_4$, 
    (iii) $F_2,F_2 \cap F_3, F_3, F_3 \cap F_4, F_4$.  
Thus, $U_i$ is convex.
  
    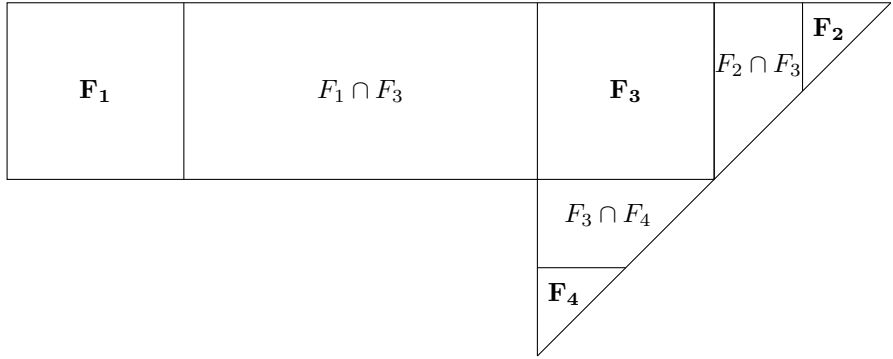
\begin{figure}[ht]
    \centering
    \begin{tikzpicture}[scale =2.35]
    \draw (0,0) rectangle (4,1);

    \draw (1,0) -- (1,1);
    \draw (3,-1) -- (3,1);
    \draw (4,0) -- (4,1);
    \draw (5,1)--(3,-1);
    \draw (4,1)--(5,1);
    \draw (4.5, 1)--(4.5,.5);
    \draw (3,-.5)--(3.5,-.5);

    \node at (.5,.5) {$\mathbf{F_1}$};
    \node at (2,.5) {$F_1 \cap F_3$};
    \node at (3.5,.5) {$\mathbf{F_3}$};
    \node at (4.25,.65) {$F_2 \cap F_3$};
    \node at (3.15,-.65) {$\mathbf{F_4}$};
    \node at (3.4, -.2) {$F_3 \cap F_4$};
    \node at (4.65, .85) {$\mathbf{F_2}$};
\end{tikzpicture}
    \caption{Convex Realization of a Case 2(b) L22 codes}
    \label{fig:case2(b)-l22}
\end{figure}
  
    \textbf{Case 3:} $F_1 \cap F_2 \cap F_3$ is a codeword of $\mincode$, but not $F_2 \cap F_3 \cap F_4$. 
    As explained earlier, the fact that 
    $F_2 \cap F_3 \cap F_4$ is not a codeword of $\mincode$ implies that 
    $\Delta((\C_4^\ast)_{\min})$ satisfies the Path-Of-Facets Condition.
    Therefore, by Lemma~\ref{lem:path}, $(\C_4^\ast)_{\min}$ has the following convex realization:
    \begin{center}
    \begin{tikzpicture}
        \filldraw[black] (2,-1) circle (2pt);
        \filldraw[black] (0,-1) circle (2pt);
        \filldraw[black] (-2,-1) circle (2pt);
        \filldraw[black] (-4,-1) circle (2pt);
        \filldraw[black] (-6,-1) circle (2pt);
        \filldraw[black] (4,-1) circle (2pt);

        \draw[thick, black] (-6,-1)--(-4,-1) node [midway, fill=white] {$\mathbf{F_{a_4}}$};
        \draw[thick, black] (-4,-1)--(-2,-1) node [midway, fill=white] {$F_{a_4} \cap F_{b_4}$};
        \draw[thick, black] (-2,-1)--(0,-1) node [midway, fill=white] {$\mathbf{F_{b_4}}$};
        \draw[thick, black] (0,-1)--(2,-1) node [midway, fill=white] {$F_{b_4} \cap F_{c_4}$};
        \draw[thick, black] (2,-1)--(4,-1) node [midway, fill=white] {$\mathbf{F_{c_4}}$};
    \end{tikzpicture}
    \end{center}
    where $(a_4, b_4, c_4)$ is a permutation of $(2,3,4)$ such that $(F_{a_4} \cap F_{b_4}) \smallsetminus F_{c_4} \neq \varnothing$, $(F_{b_4} \cap F_{c_4}) \smallsetminus F_{a_4} \neq \varnothing$, and $(F_{a_4} \cap F_{c_4}) \smallsetminus F_{b_4} = \varnothing$. As in Case 2, we 
    have $b_4=2$ or $b_4=3$.  
    By symmetry from switching the roles of $F_2$ and $F_3$, we may assume that $b_4 = 2$. 

    Next, we turn our attention to $(\C_4^\ast)_{\min}$.  
    As noted earlier, the fact that 
    $F_1 \cap F_2 \cap F_3$ is a codeword of $\mincode$ implies that 
    $\Delta((\C_1^\ast)_{\min})$ fails the Path-Of-Facets Condition. So, by definition, exactly $0$, $2$, or $3$ of the following sets are empty:
    \begin{align}
        \label{eq:apply-path-facets-L22}
        (F_1 \cap F_2) \smallsetminus F_3~,
        \quad
        (F_1 \cap F_3) \smallsetminus F_2~,
        \quad 
        (F_2 \cap F_3) \smallsetminus F_1~.        
    \end{align}
    The third set in~\eqref{eq:apply-path-facets-L22}, namely, $(F_2 \cap F_3) \smallsetminus F_1$, is nonempty because it is a superset of the nonempty triplewise intersection $F_2 \cap F_3 \cap F_4$.  Therefore, the number of empty sets in~\eqref{eq:apply-path-facets-L22} is either $0$ or $2$.  Equivalently, $ (F_1 \cap F_2) \smallsetminus F_3$ and $
        (F_1 \cap F_3) \smallsetminus F_2$
        are both nonempty (subcase 3a below) or both empty (subcase 3b).

    Before proceeding to the two subcases, we investigate the codewords appearing in $(\C_1^\ast)_{\min}$.  
    First, by construction and by definition~\eqref{eq:c-min}, $(\C_1^\ast)_{\min}$ contains the empty codeword $\varnothing$ and the maximal codewords $F_1, F_2, F_3$.  Also, $F_1 \cap F_2 \cap F_3 $ is a codeword, by Lemma~\ref{lem:path} and the fact that $\Delta((\C_1^\ast)_{\min})$ fails the Path-Of-Facets Condition.  Next, the 
    fact that $F_1\cap F_2 \cap F_3 \neq \varnothing$ and $F_2\cap F_3 \cap F_4 \neq \varnothing$, implies that $F_2 \cap F_3 \nsubseteq F_1$ and $F_2 \cap F_3 \nsubseteq F_4$; so, by Lemma~\ref{lem:intersection-2-facets}, $F_2 \cap F_3$ is a codeword of $(\C_1^\ast)_{\min}$.  
Finally, by Lemma~\ref{lem:intersection-facet}, it remains only to consider $F_1 \cap F_2$ and $F_1 \cap F_3$; we will do so in the analyses of the subcases below.

    \textit{Case 3a:} 
    Both $(F_1 \cap F_2)\smallsetminus F_3$ and $(F_1 \cap F_3)\smallsetminus F_2$ are nonempty. 
    In this case, $F_1 \cap F_2 \neq F_1 \cap F_2 \cap F_3$ and $F_1 \cap F_3 \neq F_1 \cap F_2 \cap F_3$, 
    so Lemma~\ref{lem:L22pairwise} implies that both $F_1 \cap F_2$ and $F_1 \cap F_3$ are codewords of $(\C_1^\ast)_{\min}$. 
    Thus, from the discussion above, $(\C_1^\ast)_{\min} = \{F_1, F_2, F_3, F_1 \cap F_2, F_1 \cap F_3, F_2 \cap F_3, F_1 \cap F_2 \cap F_3, \varnothing\}$. 
    A convex realization of $(\C_1^\ast)_{\min}$ is shown in the left part of Figure~\ref{fig:l22case3a} (the rectangle formed by the regions $F_3,F_2 \cap F_3, F_2$ plus the half-circle above it); this construction is due to Proposition 4.3 in~\cite{Cruz}.  ``Gluing'' this realization to the Path-Of-Facets realization of $(\C_4^\ast)_{\min}$ (shown earlier) yields 
    the realization of  $(\C^\ast_1)_{\min} \cup (\C^\ast_4)_{\min} = \C_{\min}$ shown in Figure \ref{fig:l22case3a}. 
        \begin{figure}[ht]
        \centering
    \begin{tikzpicture}[scale = 2]

    \draw(3,-1.5)--(10,-1.5);
    \draw(3,-.5)--(10,-.5);
    \draw (3,-.5) -- (7,-.5) arc(0:180:2) --cycle;


    \draw (3, -.5)--(3,-1.5);
    
    \draw (5.4, 1.46)--(4,-1.5);
    \draw (4.6, 1.46)--(6,-1.5);
    \draw (7, -.5)--(7,-1.5);
    
    \draw (9, -.5)--(9,-1.5);
    \draw (10, -.5)--(10,-1.5);


    \node at (3.5,-1) {$\mathbf{F_3}$};
    \node at (5,-1) {$F_2 \cap F_3$};
    \node at (6.5,-1) {$\mathbf{F_2}$};
    \node at (8, -1) {$F_2 \cap F_4$};
    \node at (9.5, -1) {$\mathbf{F_4}$};
    \node at (3.85, .25) {$F_1 \cap F_3$};
    \node at (5.95, .25) {$F_1 \cap F_2$};
    \node at (5, 1.1) {$\mathbf{F_1}$};
    \node at (5, -.2) {\tiny{$F_1 \cap F_2 \cap F_3$}};

\end{tikzpicture}
        \caption{Convex Realization of a Case 3a L22 code.}
        \label{fig:l22case3a}
    \end{figure}
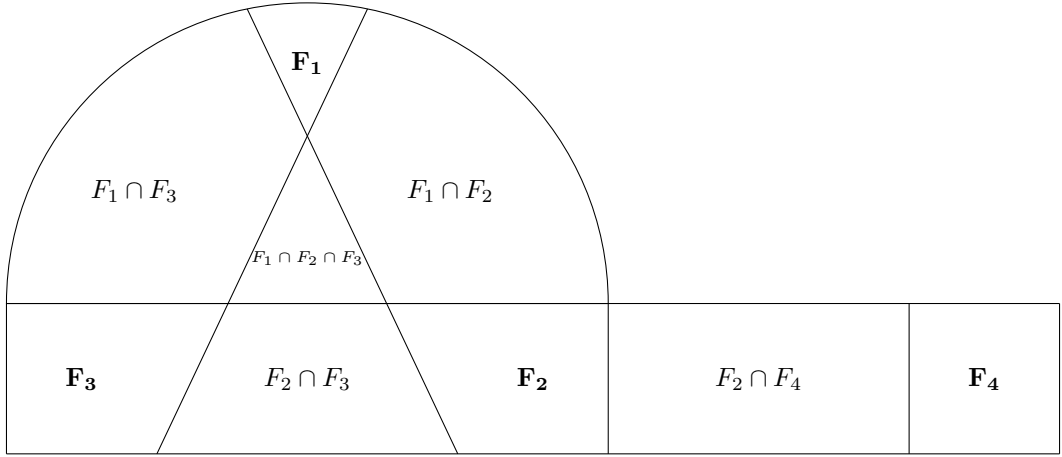

We check now that the realization is a convex realization.  For $i\in [n] \smallsetminus F_1$, the open set $U_i$ is the same as the one in the extension of the Path-of-Facets realization.  For $i\in [n] \smallsetminus F_1$, the open set $U_i$ is the same as the one in the realization of
     $(\C_1^\ast)_{\min}$ we constructed.  
     For the remaining case, $i \in F_1$, there are four possibilities: 
        (i) $i\in F_1 \smallsetminus (F_2 \cup F_3)$, 
    (ii) $i \in (F_1 \cap F_2) \smallsetminus F_3$, 
    (iii) $i \in (F_1 \cap F_3) \smallsetminus F_2$, 
    and (iv) $i \in F_1 \cap F_2 \cap F_3$. 
In these situations, the corresponding region $U_i$ is the open convex set covering the regions labeled above by, respectively, 
    (i) $F_1$, 
    (ii) $F_1,F_1 \cap F_2, F_2$, 
    (iii) $F_1,F_1 \cap F_3, F_3$, 
    (iv) $F_1, F_1 \cap F_2, F_1 \cap F_2 \cap F_3, F_1 \cap F_3, F_3, F_2 \cap F_3, F_2$.

\textit{Case 3b}: 
Both $(F_1 \cap F_2)\smallsetminus F_3$ and $(F_1 \cap F_3)\smallsetminus F_2$ are empty. In this case, $F_1 \cap F_2 = F_1 \cap F_3 = F_1 \cap F_2 \cap F_3$.  So, from the discussion immediately before Case~3a, $(\C_1^\ast)_{\min} = \{F_1, F_2, F_3, F_2 \cap F_3, F_1 \cap F_2 \cap F_3, \varnothing\}$.

We claim that the realization of $(\C^\ast_1)_{\min} \cup (\C^\ast_4)_{\min} = \C_{\min}$ shown in Figure~\ref{fig:l22case3c}
 is a convex realization.  Notice that the ``top layer'' of the realization is simply the extension of Path-of-Facets realization of $(\C_4^\ast)_{\min}$.
Thus, for $i\in [n] \smallsetminus F_1$, the open set $U_i$ is convex.       For the remaining case, $i \in F_1$, there are two possibilities: $i \in F_1 \cap F_2 = F_1 \cap F_3 = F_1 \cap F_2 \cap F_3$ and $i \in F_1 \smallsetminus (F_1 \cap F_2 \cap F_3)$.  In the first situation, the 
corresponding region $U_i$ is open rectangle encompassing the regions in Figure~\ref{fig:l22case3c} labeled by $F_3,F_2\cap F_3, F_2, F_1 \cap F_2 \cap F_3, F_1$; while in the second situation, $U_i$ is the open rectangle of the region labeled by $F_1$.  We conclude that the realization is convex, as desired.
    \begin{figure}
        \centering
        \begin{tikzpicture}[scale = 1]
            
    \draw(3,-1.5)--(10,-1.5);
    \draw(3,-.5)--(10,-.5);
    \draw (3,-1.5)--(3,-3.5)--(7,-3.5)--(7,-1.5);
    \draw (3, -2.5)--(7,-2.5);

    \draw (3, -.5)--(3,-1.5);
    
    \draw (4, -.5)--(4,-1.5);
    \draw (6, -.5)--(6,-1.5);
    \draw (7, -.5)--(7,-1.5);
    
    \draw (9, -.5)--(9,-1.5);
    \draw (10, -.5)--(10,-1.5);

    \node at (3.5,-1) {$\mathbf{F_3}$};
    \node at (5,-1) {$F_2 \cap F_3$};
    \node at (6.5,-1) {$\mathbf{F_2}$};
    \node at (8, -1) {$F_2 \cap F_4$};
    \node at (9.5, -1) {$\mathbf{F_4}$};

    \node at (5, -2) {$F_1 \cap F_2 \cap F_3$};
    \node at (5, -3) {$\mathbf{F_1}$};

        \end{tikzpicture}
        \caption{Convex realization of a Case 3b L22 code.}
        \label{fig:l22case3c}
    \end{figure}
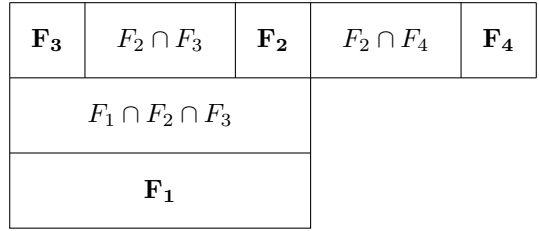

    \textbf{Case 4:} $F_1 \cap F_2 \cap F_3$ is \uline{not} a codeword of $\mincode$, but $F_2 \cap F_3 \cap F_4$ is.  This case is symmetric to Case 3, via switching the roles of $F_1$ and $F_4$. 
\end{proof}

\begin{ex}[Example~\ref{ex:nerve-is-L22} continued]
    Recall that $\C_{22} = \{\mathbf{134},\mathbf{1357},\mathbf{257},\mathbf{356}, 13, 35, 57, \varnothing\}$ 
    is an L22 code, with a convex realization shown earlier in Figure \ref{fig:runningex}.  That realization matches the one     
    shown in Figure~\ref{fig:case2(b)-l22}, with $(F_1, F_2, F_3, F_4) = (\mathbf{134}, \mathbf{356}, \mathbf{1357}, \mathbf{257})$.
\end{ex} 

\subsection{Sprockets in L24 codes} \label{sec:L24}
We have so far proven that all codes 
with nerve $\mathcal{N}(\mathcal{F}(\C))$ from L9 to L23 
are convex if and only if they do not have local obstructions. 
In the case of L24, however, this is no longer true.  Indeed, we already saw (in Example~\ref{ex:local-obstruction-L24}) that 
$\C_{24} = \{\mathbf{123},\mathbf{1246},\mathbf{145}, \mathbf{356}, 12, 14, 3 ,5 ,6, \varnothing\}$ is an L24 code that has no local obstructions
and yet is non-convex.  
Recall that this non-convexity comes from having a sprocket (Example~\ref{ex: sprocket-L24}). The following result shows that this fact extends to all L24 codes that (like $\C_{24}$) are minimal
codes for which $\Delta(\{F_1,F_2,F_3\})$ satisfies the Path-of-Facets Condition. 

\begin{thm}[L24 minimal codes] \label{thm: sprocket-L24}
    Let $\C$ be a $4$-maximal neural code that is minimal (more precisely, $\C$ is the minimal code of $\Delta(\C)$, as in~\eqref{eq:c-min}), and let
    $\mathcal{F}(\C)=\{F_1,F_2,F_3,F_4\}$ denote the set of maximal codewords.  Assume that the nerve $\mathcal{N}(\mathcal{F}(\C))$ is L24.  
    \begin{enumerate}
        \item If $\Delta(\{F_1,F_2,F_3\})$ does \uline{not} satisfy the Path-of-Facets Condition, then $\C$ is convex.
        \item If $\Delta(\{F_1,F_2,F_3\})$ satisfies the Path-of-Facets Condition, then $\C$ has a sprocket and thus is \uline{not} convex. 
    \end{enumerate}
\end{thm}

\begin{proof}
Part~(1) follows directly from Proposition~\ref{prop: pathoffaces-convex}.  
For part~(2), assume that 
$\Delta(\{F_1,F_2,F_3\})$ satisfies the Path-of-Facets Condition, which means that 
exactly one of the following three sets is empty: 
$(F_1 \cap F_2) \smallsetminus F_3$, $(F_1 \cap F_3) \smallsetminus F_2$, and $(F_2 \cap F_3) \smallsetminus F_1$.
However, the simplicial complex L24 is symmetric in $F_1,F_2,F_3$, so we can relabel $F_1,F_2,F_3$, if needed, so that $(F_1 \cap F_3) \smallsetminus F_2$ is empty but the other two sets are nonempty. The situation is summarized in Figure~\ref{fig:L24-proof}, which also labels the following sets:
\begin{align} \label{eq:proof-L24-intersections}
\sigma_i ~:=~ F_i \cap F_4 \quad \mathrm{(for }~ i=1,2,3)~,
    \quad  \tau ~:=~ F_1 \cap F_2 \cap F_3~,\quad  
    \mathrm{and} \quad  \rho_j~:=~F_j \cap F_2 \quad \mathrm{(for}~ j=1,3)~.  
\end{align} 
All of these sets $\sigma_i$, $\tau$, and $\rho_j$ are nonempty, due to the fact that $\C$ is an L24 code.

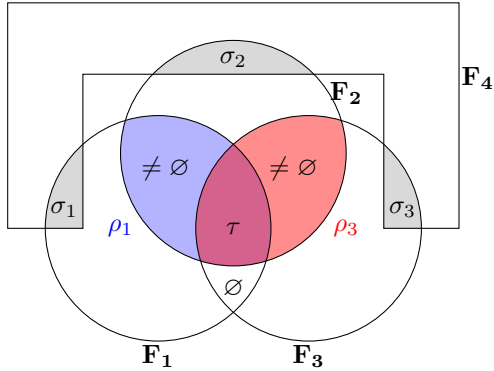
\begin{figure}[ht]
    
\begin{tikzpicture}
 

\begin{scope}
    \clip (0,0) circle(1.5cm);
    \clip (1,1) circle(1.5cm);
    \fill[blue,opacity=0.30](0,0) circle(1.5cm);
\end{scope}

\begin{scope}
    \clip (2,0) circle(1.5cm);
    \clip (1,1) circle(1.5cm);mm
    \fill[red,opacity=0.45](2,0) circle(1.5cm);
\end{scope}

\begin{scope}
    \clip (-2,0) rectangle (-1,1.5);
    \clip (0,0) circle(1.5cm);
    \fill[opacity=0.15](0,0) circle(1.5cm);
\end{scope}

\begin{scope}
    \clip (3,0) rectangle (4,1.5);
    \clip (2,0) circle(1.5cm);
    \fill[opacity=0.15](2,0) circle(1.5cm);
\end{scope}

\begin{scope}
    \clip (-1,2.05) rectangle (3,3);
    \clip (1,1) circle(1.5cm);
    \fill[opacity=0.15](1,1) circle(1.5cm);
\end{scope}

\draw (0,0) circle(1.5cm);
\draw (2,0) circle(1.5cm);
\draw (1,1) circle(1.5cm);

\draw (-1,0) -- (-2,0) -- (-2,3) -- (4,3) -- (4,0) --(3,0) -- (3,2.05) -- (-1,2.05) -- cycle;
  
\node at (0,-1.7) {$\mathbf{F_1}$};
\node at (2,-1.7) {$\mathbf{F_3}$};
\node at (2.5,1.8) {$\mathbf{F_2}$};
\node at (4.25,2) {$\mathbf{F_4}$};
\node at (1,-0.8) {$\varnothing$};
\node at (0.1,0.8) {$\neq \varnothing$};
\node at (1.8,0.8) {$\neq \varnothing$};
\node at (1,0) {$\tau$};
\node[blue] at (-0.5,0) {$\rho_1$};
\node[red] at (2.5,0) {$\rho_3$};
\node at (-1.25,0.25) {$\sigma_1$};
\node at (1,2.25) {$\sigma_2$};
\node at (3.25,0.25) {$\sigma_3$};

\end{tikzpicture}

\caption{ 
A Venn diagram of the maximal codewords $F_1,F_2,F_3,F_4$ of any code $\C$ as in the proof of Theorem~\ref{thm: sprocket-L24} (that is, 
any L24 minimal code $\C$ for which $\Delta(\{F_1,F_2,F_3\})$ satisfies the Path-of-Facets Condition with $(F_1 \cap F_3) \smallsetminus F_2 \neq \varnothing$).  The fact that $F_4$ intersects each of $F_1$, $F_2$, and $F_3$ individually -- but does not intersect $F_1\cap F_2$, $F_1 \cap F_3$, or $F_2 \cap F_3$ -- comes from the fact that $\C$ is an L24 code (see Figure~\ref{fig:image-simplicial-complex}).  Certain intersections are labeled in this Venn diagram, namely, 
$\sigma_i$, $\tau$, and $\rho_j$, as defined in~\eqref{eq:proof-L24-intersections}.}
\label{fig:L24-proof}
\end{figure}
    
Next, we claim that the code $\C$ consists of $10$ codewords, as follows:
    \begin{align}
        \label{eq:code-L24-proof}
    \C~=~\{ \mathbf{F_1},~ \mathbf{F_2},~ \mathbf{F_3},~ \mathbf{F_4},~ \rho_1=F_1 \cap F_2,~ \rho_3=F_2\cap F_3,~ \sigma_1=F_1\cap F_4, ~
\sigma_2=F_2\cap F_4,~ 
\sigma_3=F_3\cap F_4,~ \varnothing \}~.        
    \end{align}
Indeed, each of the facet-intersections $\rho_i$ and $\sigma_j$ is seen from Figure~\ref{fig:L24-proof} to be the intersection of exactly two facets $F_i$, and so (by Lemma~\ref{lem:intersection-2-facets}) is a mandatory face of $\Delta(\C)$.  On the other hand, the remaining intersection of facets, $\tau = F_1 \cap F_2 \cap F_3$, is non-mandatory (this is essentially Lemma~\ref{lem:path}(a)).  Now our claim follows readily from~\eqref{eq:c-min} and Lemma~\ref{lem:intersection-facet}.

From~\eqref{eq:code-L24-proof}, we obtain the Hasse diagram of $\C$ shown in Figure~\ref{fig:hasse-diagram-L24}. This figure will be useful in 
the rest of the proof.

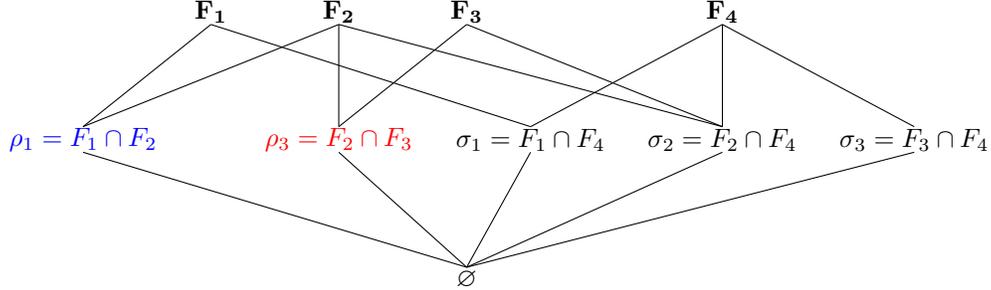
\begin{figure}[ht]
\begin{tikzpicture}[scale=1.7]
 

\node at (-2,3) {$\mathbf{F_1}$};
\node at (-1,3) {$\mathbf{F_2}$};
\node at (0,3) {$\mathbf{F_3}$};
\node at (2,3) {$\mathbf{F_4}$};
\node[blue] at (-3,2) {$\rho_1=F_1\cap F_2$};
\node[red] at (-1,2) {$\rho_3=F_2\cap F_3$};
\node at (0.5,2) {$\sigma_1=F_1\cap F_4$};
\node at (2,2) {$\sigma_2=F_2\cap F_4$};
\node at (3.5,2) {$\sigma_3=F_3\cap F_4$};
\node at (0,0.9) {$\varnothing$};

\draw (0,1) -- (-3,1.9);
\draw (0,1) -- (-1,1.9);
\draw (0,1) -- (0.5,1.9);
\draw (0,1) -- (2,1.9);
\draw (0,1) -- (3.5,1.9);

\draw (-3,2.1) -- (-2,2.9);
\draw (-3,2.1) -- (-1,2.9);

\draw (-1,2.1) -- (-1,2.9);
\draw (-1,2.1) -- (0,2.9);

\draw (0.5,2.1) -- (-2,2.9);
\draw (0.5,2.1) -- (2,2.9);

\draw (2,2.1) -- (-1,2.9);
\draw (2,2.1) -- (2,2.9);

\draw (2,2.1) -- (0,2.9);
\draw (3.5,2.1) -- (2,2.9);

\end{tikzpicture}

    \caption{The Hasse diagram (with respect to inclusion) of the codewords in any code $\C$ as in the proof of Theorem~\ref{thm: sprocket-L24}.  This means that a line is drawn from one codeword $c$ up to another codeword $c'$ if $c \subseteq c'$ and there is no codeword $d \in \C$ for which $c \subsetneq d \subsetneq c'$.}
    \label{fig:hasse-diagram-L24}
\end{figure}
    
    Our aim is to show that $(\sigma_1, \sigma_2, \sigma_3, \tau)$ is a sprocket of $\C$ (as in Definition \ref{def:wheel-sprocket}). 
    We begin by showing that $(\sigma_1, \sigma_2, \sigma_3, \tau)$ is a partial-wheel. The first requirement P(i) is that 
    $\sigma_1 \cup \sigma_2 \cup \sigma_3 \in \Delta(\C)$ and 
    $\Tk_{\C}(\sigma_i \cup \sigma_j) = \Tk_{\C}(\sigma_1 \cup \sigma_2 \cup \sigma_3)$ for every $1 \leq i < j \leq 3$.  
    By construction, $\sigma_1 \cup \sigma_2 \cup \sigma_3$ is a subset of $F_4$ (this fact 
    is also readily seen from Figure~\ref{fig:L24-proof}), and $F_4$ is a facet of $\Delta(\C)$;  therefore $\sigma_1 \cup \sigma_2 \cup \sigma_3$ is a face of $\Delta(\C)$, as required.  
    Next, it is readily seen from the Hasse digram in Figure~\ref{fig:hasse-diagram-L24} that 
        $\Tk_{\C}(\sigma_i \cup \sigma_j) = \Tk_{\C}(\sigma_1 \cup \sigma_2 \cup \sigma_3) = \{F_4\}$ for all $1 \leq i < j \leq 3$.  We conclude that P(i) holds.

    Next, P(ii) requires that $\sigma_1 \cup \sigma_2 \cup \sigma_3 \cup \tau \notin \Delta(\C)$. 
    From Figure~\ref{fig:L24-proof}, we see that $\sigma_1 \cup \sigma_2 \cup \sigma_3 \cup \tau $ is not a subset of any facet $F_i$ of $\Delta(\C)$ and therefore is not a face of $\Delta(\C)$, as desired.  

    Finally, for  P(iii)$_\circ$, we must have that $\sigma_i \cup \tau \in \Delta(\C)$, for $i=1,2,3$. 
    Figure~\ref{fig:L24-proof} shows that 
    $ \sigma_i \cup \tau$ is a subset of $ F_i$, which is a facet of $\Delta(\C)$, and so 
    $\sigma_i \cup \tau \in \Delta(\C)$.
Therefore $(\sigma_1, \sigma_2, \sigma_3, \tau)$ is a partial-wheel of $\C$.

    We now show that the partial-wheel  $(\sigma_1, \sigma_2, \sigma_3, \tau)$ is, in fact,  a sprocket with witnesses $\rho_1$ and $\rho_3$.  
    Following Definition \ref{def:wheel-sprocket}, we must verify three additional conditions:
    \begin{enumerate}[start=1,label={\bfseries{S}(\arabic*):}, itemindent=5pt]
        \item $\Tk_{\C}(\sigma_j \cup \tau)\subseteq \Tk_{\C}(\rho_j)$ for $j \in \{1,3\}$,
            
        \item $\Tk_{\C}(\tau) \subseteq \Tk_{\C}(\rho_1) \cup \Tk_{\C}(\rho_3)$, and 
            
        \item $\Tk_{\C}(\rho_1 \cup \rho_3 \cup \tau) \subseteq \Tk_{\C}(\sigma_2)$.
            
    \end{enumerate}
    Using Figures~\ref{fig:L24-proof} and~\ref{fig:hasse-diagram-L24}, we compute the trunks appearing above, and readily verify the three conditions:
    \begin{itemize}
        \item $\Tk_{\C}(\sigma_j \cup \tau) = \{F_j\}$, for $j \in \{1,3\}$,
        \item $\Tk_{\C}(\rho_j) = \{\rho_j, F_j\}$, for $j \in \{1,3\}$
            
        \item $\Tk_{\C}(\tau) = \{ \rho_1, \rho_3, F_1,F_2,F_3\}$, 
        \item $ \Tk_{\C}(\rho_1) \cup \Tk_{\C}(\rho_3) = \{F_2\}$,  and
        \item $ \Tk_{\C}(\sigma_2) = \{\sigma_2,F_2,F_4\}$.
    \end{itemize}
    
    Finally, the fact that $\C$ has a sprocket implies (by Lemma~\ref{lem:sprockets-nonconvex}) that $\C$ cannot be convex.
\end{proof}

\begin{ex}[Example~\ref{ex:nerve-is-L24} continued; L24 code]  \label{ex: sprocket-min-L24}
Recall (from Examples~\ref{ex:local-obstruction-L24},~\ref{ex:path-facets-L24}, 
and~\ref{ex:nerve-is-L24}) that 
$\C_{24} =  \{\mathbf{123},\mathbf{1246},\mathbf{145}, \mathbf{356}, 12, 14, 3 ,5 ,6, \varnothing\}$
is a minimal L24 code; and that $\Delta(\{ F_1,F_2,F_3 \})$ satisfies the Path-of-Facets Condition (here, $F_1=123$, $F_2=1246$, and $F_3=145$). 
For this code $\C_{24}$,
the sprocket given in the proof of Theorem~\ref{thm: sprocket-L24} matches the one shown earlier in Example~\ref{ex: sprocket-L24}.
\end{ex}

\begin{ex}
    Consider the code $\C = \{\mathbf{1236}, \mathbf{234}, \mathbf{135}, \mathbf{456}, 13, 23, 4, 5, 6, \varnothing\}$, which is the code called 
    ``$\C_2$'' in prior works \cite{sunflower, wheels}.
    It is straightforward to check that $\C$ is an L24 code. It was first shown in \cite{sunflower} that $\C$ is non-convex; and a sprocket of $\C$ is constructed in~\cite[Example~4.7]{wheels}. 
    The sprocket in that example is exactly the sprocket described in Theorem~\ref{thm: sprocket-L24}.
\end{ex}

Theorem~\ref{thm: sprocket-L24} resolves 
Conjecture~\ref{conj-Jeffs} for minimal L24 codes, but 
the conjecture remains open for L24 codes that are non-minimal.  We illustrate this in the next example.

\begin{ex}
[Example~\ref{ex: sprocket-min-L24} continued]  \label{ex:min-L24-plus-more-codewords}
We revisit the minimal L24 code \[\C_{24} =  \{\mathbf{123},\mathbf{1246},\mathbf{145}, \mathbf{356}, 12, 14, 3 ,5 ,6, \varnothing\}.\]  Consider adding codewords to $\C_{24}$, while maintaining the same simplicial complex.  If we add the codeword $1$ (and any additional codewords), the resulting L24 code is max-intersection-complete and therefore convex.  However, if we do not add $1$, then the convexity status of the resulting L24 code is unclear.
%
\end{ex}

\begin{remark}[L28 codes] \label{rem:L28}
    Consider an L24 code $\mathcal{D}$ on $n$ neurons that satisfies the hypotheses of Theorem~\ref{thm: sprocket-L24} (for instance, $\mathcal{D} = \code_{24}$, where $n=6$).  Let $\mathcal{D}_{28}$ be the code obtained from $\mathcal{D}$ by adding a new neuron $n+1$ to all nonempty codewords (for instance, if $\mathcal{D} = \code_{24}$, then $\mathcal{D}_{28}$ has maximal codewords $\mathbf{1237},\mathbf{12467},\mathbf{1457}, \mathbf{3567}$).  It is straightforward to check that $\mathcal{D}_{28}$ is a minimal L28 code (this essentially follows from two facts: there is a bijection between the maximal codewords of $\mathcal{D}$ and those of $\mathcal{D}_{28}$, and the links of intersections of maximal codewords are the same between the two codes). Additionally, $\mathcal{D}_{28}$ has a sprocket (the ideas in the proof of Theorem~\ref{thm: sprocket-L24} extend to apply to $\mathcal{D}_{28}$) and so is non-convex.  We conclude that Theorem~\ref{thm:L24-summary}(2) has an analogue obtained by replacing ``L24'' by ``L28''.
\end{remark}

\begin{remark} \label{rem:wheels-table-3-codes}
The authors of~\cite{wheels} enumerated $6$-neuron minimal codes with up to seven maximal codewords.  Among such codes with exactly four maximal codewords, three were found to have wheels (in fact, sprockets)~\cite[Table~1]{wheels}.  
Two of these codes are L25 and L28 codes, while the third one, shown below, is an L24 code that 
satisfies the hypotheses of Theorem~\ref{thm: sprocket-L24}:
\[
\{
\mathbf{123},\mathbf{145},\mathbf{246},\mathbf{1356},13, 15,2,4,6,
\varnothing
\}~.
\]

\end{remark}

\begin{remark}
    It is known that
    if a neural code $\C$ has a sprocket, then $\C$ has a sprocket $(\sigma_1, \sigma_2, \sigma_3, \tau)$ such that $\tau \notin \C$ and $\tau$ is a max-intersection face of $\Delta(\C)$~\cite[Corollary~5.7]{wheels}. 
    For codes satisfying the hypotheses of Theorem~\ref{thm: sprocket-L24}, such a sprocket is given in the proof of the theorem (we have $\tau = F_1 \cap F_2 \cap F_3 \notin \C$). 
    More generally, it is conjectured that (1) if $(\sigma_1, \sigma_2, \sigma_3, \tau)$ is a wheel of a code $\C$, then $\tau \notin \C$, and (2) if a wheel of the code $\C$ exists, then there exists a wheel $(\sigma_1, \sigma_2, \sigma_3, \tau)$ in which $\tau$ is a max-intersection face~\cite[Conjecture~5.10]{wheels}.
\end{remark}

\section{Discussion} \label{sec:discussion}
We return to the motivation for this work, Conjecture~\ref{conj-Jeffs}, which states that $4$-maximal codes are convex if and only if they have no local obstructions and no wheels.  Our work made significant progress toward resolving this conjecture.  Specifically, we proved that it holds for all codes from L9 to L23 (Corollary~\ref{cor:conjecture}); in fact, for such codes, convexity is equivalent to having no local obstructions (there is no need to check for wheels).  We also resolved Conjecture~\ref{conj-Jeffs} for L24 codes that are minimal (Theorem~\ref{thm:L24-summary}), and in this case wheels (specifically, sprockets) are needed.     
The remaining open cases of Conjecture~\ref{conj-Jeffs} are non-minimal  L24 codes (recall Example~\ref{ex:min-L24-plus-more-codewords}) and all cases of L25 through L28.  Additionally, we do not know what types of wheels occur in such codes (for instance, are there any beyond sprockets?).

Our work points the way toward future progress on codes with more than four maximal codewords.  We know, for instance, that Conjecture~\ref{conj-Jeffs} extends to all codes for which all triplewise intersections of maximal codewords are empty (Proposition~\ref{prop:notriangle}).  Additionally, we expect that ideas developed in this work -- gluing realizations together, in particular -- will be useful.  
Finally, we also know that the convex realizations will no longer always be planar; indeed, there is a $5$-maximal convex code -- namely, the ``sunflower code''~$\mathcal{S}_4$~\cite[Proposition~2.3]{jeffsdim} -- that 
has no convex realization even in $\mathbb{R}^3$ (but has one in $\mathbb{R}^4$).

\subsubsection*{Acknowledgments} {\footnotesize This research was initiated by Gisel Flores, Osiano Isekenegbe, and Deanna Perez in the 2022 MSRI-UP REU, which was hosted by the Mathematical Sciences Research Institute (MSRI, now SL-Math) and was supported by the NSF (DMS-2149642) and the Alfred P. Sloan Foundation. Saber Ahmed, Natasha Crepeau, and Anne Shiu were the mentors of MSRI-UP for this project.
Anne Shiu was partially supported by the NSF (DMS-1752672).
The authors are grateful to Federico Ardila for his leadership as MSRI-UP's on-site director and for his helpful conversations. 
The authors also thank  
Scribble Together for proving tools for remote collaboration.}

\bibliographystyle{plain}
\bibliography{ref.bib}

\begin{thebibliography}{10}

\bibitem{axis-parallel}
Miguel Benitez, Siran Chen, Tianhui Han, R.~Amzi Jeffs, Kinapal Paguyo, and Kevin~A. Zhou.
\newblock Realizing convex codes with axis-parallel boxes.
\newblock {\em Involve}, 17:633---649, 2024.

\bibitem{planarcodes-decidable}
Boris Bukh and R.~Amzi Jeffs.
\newblock Planar convex codes are decidable.
\newblock {\em SIAM J.\ Discrete Math.}, 37(2):951--963, 2023.

\bibitem{nondegen}
Patrick Chan, Katherine Johnston, Joseph Lent, Alexander~Ruys de~Perez, and Anne Shiu.
\newblock Nondegenerate neural codes and obstructions to closed-convexity.
\newblock {\em SIAM J.\ Discrete Math.}, 37(1):114--145, 2023.

\bibitem{decidability}
Aaron Chen, Florian Frick, and Anne Shiu.
\newblock Neural codes, decidability, and a new local obstruction to convexity.
\newblock {\em SIAM J.\ Appl.\ Algebra Geom.}, 3(1):44--66, 2019.

\bibitem{Cruz}
Joshua Cruz, Chad Giusti, Vladimir Itskov, and Bill Kronholm.
\newblock On open and closed convex codes.
\newblock {\em Discrete Comput.\ Geom.}, 61:247---270, 2018.

\bibitem{Curto}
Carina Curto, Elizabeth Gross, Jack Jeffries, Katherine Morrison, Mohamed Omar, Zvi Rosen, Anne Shiu, and Nora Youngs.
\newblock What makes a neural code convex?
\newblock {\em SIAM J.\ Appl.\ Algebra Geom.}, 1:222--238, 2017.

\bibitem{Giusti}
Chad Giusti and Vladimir Itskov.
\newblock A no-go theorem for one-layer feedforward networks.
\newblock {\em Neural Comput.}, 26:2527---2540, 2014.

\bibitem{sunflower}
R.~Amzi Jeffs.
\newblock Sunflowers of convex open sets.
\newblock {\em Adv.\ Appl.\ Math.}, 111(1), 2019.

\bibitem{morphisms}
R.~Amzi Jeffs.
\newblock Morphisms of neural codes.
\newblock {\em SIAM J.\ Appl.\ Algebra Geom.}, 4(1):99--122, 2020.

\bibitem{jeffsdim}
R.~Amzi Jeffs.
\newblock Open, closed, and non-degenerate embedding dimensions of neural codes.
\newblock {\em Discrete Comput.\ Geom.}, 71:764--786, 2023.

\bibitem{embedding-dim-gaps}
R.~Amzi Jeffs, Henry Siegel, David Staudinger, and Yiqing Wang.
\newblock Embedding dimension gaps in sparse codes.
\newblock Preprint available at \href{https://arxiv.org/abs/2309.14862}{{\tt ar{X}iv:2309.14862}}, 2023.

\bibitem{Shiu}
Katherine Johnston, Anne Shiu, and Clare Spinner.
\newblock Neural codes with three maximal codewords: Convexity and minimal embedding dimension.
\newblock {\em Involve}, 15:333---343, 2022.

\bibitem{matroids}
Alexander~B Kunin, Caitlin Lienkaemper, and Zvi Rosen.
\newblock Oriented matroids and combinatorial neural codes.
\newblock {\em Comb.\ Theory}, 3(1), 2023.

\bibitem{Lienkaemper}
Caitlin Lienkaemper, Anne Shiu, and Zev Woodstock.
\newblock Obstructions to convexity in neural codes.
\newblock {\em Adv.\ Appl.\ Math.}, 85:31---59, 2017.

\bibitem{graphs-neural}
Suhith~K N and Neha Gupta.
\newblock Properties of graphs of neural codes.
\newblock Preprint available at \href{https://arxiv.org/abs/2403.17548}{{\tt ar{X}iv:2403.17548}}, 2024.

\bibitem{wheels}
Alexander {Ruys de Perez}, Laura~Felicia Matusevich, and Anne Shiu.
\newblock Wheels: A new criterion for non-convexity of neural codes.
\newblock {\em Adv.\ Appl.\ Math.}, 150:102567, 2023.

\end{thebibliography}

\end{document}